\documentclass[reqno,11pt]{amsart}
\usepackage{amssymb,amscd,pxfonts,mathrsfs,epsfig,accents} 
\theoremstyle{plain}
        \newtheorem{theorem}{Theorem}[section]
         
        \newtheorem{lemma}[theorem]{Lemma}

        \newtheorem{remark}[theorem]{Remark}  
\numberwithin{equation}{section}

\newcommand \Rm  {\text{Riem}}
\newcommand \MM    {{\mathcal M}}
\newcommand \imm {\psi}
\newcommand \rig {\ell}
\newcommand \immbis {{\tilde{\imm}}}

\newcommand \gbis   {{\tilde g}}
\newcommand \HH  {{\mathcal H}}
\newcommand \nH  {\nu} 
\newcommand \gH  {\undertilde{g}}
\newcommand \naH {\undertilde{\nabla}} 
\newcommand \KH  {K}
\newcommand \LH  {L}
\newcommand \MH  {M}
\newcommand \gHbis  {\tilde{\gH}}
\newcommand \naHbis {{\tilde{\naH}}} 
\newcommand \KHbis  {{\tilde{\KH}}}
\newcommand \LHbis  {{\tilde{\LH}}}
\newcommand \MHbis  {{\tilde{\MH}}}
\newcommand \Mink  {{\mathbb M}^{n+1}} 
\newcommand \idMink  {I_\gMink}
\newcommand \gMink {\eta}
\newcommand \naMink  {\nabla^\eta}
\newcommand \rigMink {\rig'}
\newcommand \rigMinkbis {\tilde\rigMink}
\newcommand \naHMink {\undertilde{\nabla}'} 
\newcommand \KHMink  {K'}
\newcommand \LHMink  {L'}
\newcommand \MHMink  {M'}
\newcommand \Dcal 		{\mathcal{D}} 
\newcommand \Fcal 		{\mathcal{F}}

\newcommand \Acal 		{\mathcal{A}}

\newcommand \RR 		{\mathbb{R}}  
\newcommand \LL 		{\mathbb{L}} 
\newcommand \OO 		{\mathbb{O}} 
\newcommand \Sym 		{\mathbb{S}}  
\newcommand \del  	        \partial
\newcommand \eps 	        \epsilon
\newcommand \lam    	        \lambda
\newcommand \gam 	        \gamma
\newcommand \Gam                \Gamma

\newcommand \la			\langle
\newcommand \ra			\rangle

\def \ov   {\overline}
\def \wt   {\tilde}

\def \d    {\partial}
\def \al   {\alpha}
\def \be    {\beta}

\def \De   {\Delta}
\def \ep   {\varepsilon}
\def \ga   {\gamma}
\def \Ga   {\Gamma}

\def \Om   {\Omega}

\def \ph   {\varphi}

\def \si   {\sigma}
\def \ta   {\tau}

\def \e    {\mathbb{E}}

\def \n    {\mathbb{N}}
\def \r    {\mathbb{R}}

\def \C    {\mathcal{C}} 
\def \D    {\mathcal{D}} 
\def \loc  {\textit{\!loc}}
\def \diag {\textup{diag}}

\newcommand{\mat}[4] {{\left(\aligned  & #1 && #2 \\  & #3 && #4 \endaligned\right)}}

\begin{document}
\title[Isometric immersions into the Minkowski spacetime]
{Isometric immersions into the Minkowski spacetime
\\
for Lorentzian manifolds with limited regularity}
\author
   [P.G. L{\tiny e}Floch, C. Mardare \and S. Mardare]
    {Philippe G. L{\tiny e}Floch, Cristinel Mardare, Sorin Mardare}
\address
   {P. G. LeFloch and C. Mardare\\
   Laboratoire Jacques-Louis Lions\\ 
   \& Centre National de la Recherche Scientifique\\
   Universit\'e de Paris VI, 4 Place Jussieu\\ 
   75252 Paris, France.}  
\email{LeFloch@ann.jussieu.fr, Mardare@ann.jussieu.fr}
\address
   {S. Mardare\\
    Instit\"ut f\"ur Mathematik\\
    Abt. Angewandte Mathematik\\
    Universit\"at Z\"urich\\
    Winterthurerstrasse 190\\
    8057 Z\"urich, Switzerland.}
\email{sorin.mardare@math.unizh.ch.}
\date{\today}


\begin{abstract} 
Assuming minimal regularity assumptions on the data, we revisit the classical problem of finding 
isometric immersions into the Minkowski spacetime for hypersurfaces of a Lorentzian manifold. 
Our approach encompasses metrics having Sobolev regularity and Riemann curvature defined 
in the distributional sense, only. 
It applies to timelike, spacelike, or null hypersurfaces with arbitrary signature that possibly changes
from point to point. 
\end{abstract}
\date{November 1, 2007. \, Submitted to : \it Discrete and Continuous Dynamical Systems.}
\maketitle


\section{Introduction} 
\label{intro}

 Given a Lorentzian manifold $(\MM,g)$ of dimension $n+1$ or a hypersurface $\HH\subset \MM$, we study here whether it can be immersed isometrically in the Minkowski space $\Mink:=(\r^{n+1}, \gMink)$. While this subject has been extensively studied within the class of smooth immersions in the context of Riemannian geometry, we are interested in the present paper in the case of Lorentzian manifolds and their hypersurfaces of arbitrary signature and in metrics with limited regularity in a Sobolev space. 
 Our analysis will encompass metrics $g$ of class $W^{1,p}_\loc$ with $p$ greater than the dimension of the underlying manifold, which, in fact, is the optimal regularity. We prove the existence of a global isometric immersion if the underlying manifold is simply connected. We also prove the uniqueness up to isometries of the Minkowski space and the stability of the immersion with respect to the metric.

One of our results is as follows: 


\begin{theorem}[Immersion of a manifold in Minkowski spacetime] 
\label{main-one}
Let $(\MM,g)$ be a simply connected, Lorentzian manifold with dimension $n+1$ whose metric $g$ is of class $W^{1,p}_\loc(\MM)$ with $p>n+1$. 
Then, there exists an isometric immersion $\imm:\MM \to \Mink$ of class $W^{2,p}_\loc(\MM)$ if and only if the Riemann curvature (tensor field) $\Rm_g$ defined in the distributional sense vanishes. 
Furthermore, the application $g\mapsto \imm$ is locally Lipschitz continuous in the following sense. For any connected open set $\Acal\Subset \MM$ and any $\ep>0$, there exists a constant $C(\ep,\Acal)$ with the following property: if $g$ and $\gbis$ are metrics on $\MM$ that satisfy 
$$
\min \big(|\det g| , |\det\gbis| \big) \geq \ep, 
\qquad 
\max \big(\|g\|_{W^{1,p}(\Acal)},\|\gbis\|_{W^{1,p}(\Acal)}\big) \leq \frac1\ep,
$$
then there exists isometries $\pi, \wt\pi:\Mink\to\Mink$ of the Minkowski space such that the corresponding isometric immersions $\imm: (\MM,g) \to \Mink$ and $\immbis: (\MM,\gbis) \to \Mink$ satisfy the inequality
$$
\|\wt\pi\circ\immbis-\pi\circ\imm\|_{W^{2,p}(A)}\leq C(\ep,\Acal) \, \|g-\gbis\|_{W^{1,p}(A)}. 
$$ 
\end{theorem}

 The corresponding immersion problem in the case of the Euclidian space has been recently revisited by Ciarlet and his collaborators; see, for instance, \cite{Ciarlet} and the references therein. 
 Observe that, in the above theorem, the curvature of the manifold is defined in the sense of distributions only; for the definition of covariant derivatives and curvature tensors associated with metrics with limited regularity we rely on
LeFloch and C. Mardare \cite{lfma} and the references cited therein. 
The proof of Theorem~\ref{main-one} (in Section \ref{manifold} below) will rely on earlier work by S. Mardare 
on Pfaff-type systems \cite{sor-infty,sor-lpN}. Previous arguments strongly used the assumption that the metric under consideration 
was Riemannian. To establish Theorem~\ref{main-one}, we take account that the metric is Lorentzian; in fact, our argument 
immediately extends also to any pseudo-Riemannian manifold (with arbitrary signature). 

Our second contribution concerns the immersion of {\sl hypersurfaces}
within the Minkowski space. We first consider the case of hypersurfaces with general signature, 
then we specialize our results to spacelike or timelike submanifolds. 
Consider a hypersurface $\HH\subset \MM$ in a Lorentzian manifold with dimension $n+1$ and a transverse field (henceforth called rigging) $\rig$ along $\HH$, that is, a vector field $\rig\in T\MM$ that is transversal to $\HH$. Then, 
the Levi-Civita connection of $\MM$ can be decomposed into ``tangent'' and ``transversal'' components as follow:
$$
\aligned 
\nabla_X Y & =\naH_X Y - \KH(X,Y)\rig,  \quad X, Y\in T\HH,\\
\nabla_X \rig & =\LH(X) - \MH(X)\rig, \quad X\in T\HH, 
\endaligned
$$
where $\naH,\KH,\LH,\MH$ are operators defined on $T\HH$. We say that an immersion $\imm:\HH\to\Mink$ and a rigging $\rigMink:\HH\to T\Mink$ {\sl preserve the operators}
 $\naH,\KH,\LH,\MH$ if the Levi-Civita connection $\naMink$ of the Minkowski spacetime satisfies 
$$
\aligned
\naMink_{\imm_*X} \imm_*Y   & =\imm_*(\naH_{X} Y) - \KH(X,Y)\rigMink, \qquad X,Y\in T\HH,
\\
\naMink_{\imm_*X} \rigMink & =\imm_*(\LH(X))  - \MH(X)\rigMink, \qquad X\in T\HH.
\endaligned
$$

The main result established in Section \ref{hypersurface1} below is as follows.


\begin{theorem}[Immersion of a hypersurface with rigging] 
\label{main-two}
With the notation above, suppose that $\HH$ is simply connected and the operators $\naH,\KH,\LH,\MH$ are of class $L^p_\loc(\HH)$ with $p>n$.
Then, there exists an immersion $\imm:\HH \to \Mink$ and a rigging $\rigMink:\HH\to T\Mink$, respectively of class $W^{2,p}_\loc(\HH)$ and $W^{1,p}_\loc(\HH)$, preserving these operators if and only if the (generalized) 
Gauss and Codazzi equations (see \eqref{gc} below) are satisfied. 
Moreover, the application $(\naH,\KH,\LH,\MH) \mapsto (\imm,\rigMink)$
is locally Lipschitz continuous in the following sense. 
For any connected open set $\Acal \Subset \HH$ and any $\ep>0$, there exists a constant $C(\ep,\Acal)$ with the following property: 
if two sets of operators $\naH,\KH,\LH,\MH$ and $\naHbis,\KHbis,\LHbis,\MHbis$ satisfy 
$$
\max\big(\|(\naH,\KH,\LH,\MH)\|_{L^p(\Acal)},\|(\naHbis,\KHbis,\LHbis, \MHbis)\|_{L^p(\Acal)}\big)  \leq \frac1\ep, 
$$
then there exists an affine bijection $\sigma:\Mink\to\Mink$ of the Minkowski spacetime such that the corresponding immersions $\imm,\immbis:\HH\to\Mink$ and riggings $\rigMink,\rigMinkbis$ satisfy the inequality
$$
\aligned 
& \|\immbis-\sigma\circ\imm\|_{W^{2,p}(\Acal)}+\|\rigMinkbis -\sigma_*\rigMink\|_{W^{1,p}(\Acal)} 
\\
& \leq C(\ep,\Acal) \, \Big( 
\|(\naHbis,\KHbis,\LHbis, \MHbis) - (\naH,\KH,\LH,\MH)\|_{L^p(\Acal)}   
\Big).
\endaligned 
$$ 
\end{theorem}

Observe that no assumption is made on the signature of the hypersurface. 
In particular, this theorem applies to hypersurfaces that are nowhere null. For such hypersurfaces we 
obtain in Section \ref{hypersurface2} a simpler derivation by choosing 
 the rigging to be the unit normal vector field to $\HH$. Recall that 
 the pull-back on $\HH$ of the forms $g$ and $\nabla n$, denoted by
 $\gH$ and $\KH$, are the first and second fundamental forms of $\HH\subset \MM$, respectively.
 In Section \ref{hypersurface2} below we will prove:


\begin{theorem}[Immersion of spacelike or timelike hypersurfaces] 
\label{main-three}
Suppose that $\HH$ is simply connected and nowhere null and that $(\gH,\KH)$ is
 of class $W^{1,p}_\loc(\HH)\times L^p_\loc(\HH)$ with $p>n$. 
Then, there exists an immersion $\imm:\HH \to \Mink$ of class $W^{2,p}_\loc(\HH)$ preserving the fundamental forms  $\gH$ and $\KH$,
if and only if the Gauss and Codazzi equations (see \eqref{gc+} below) 
are satisfied. Moreover, the application $(\gH,\KH) \mapsto \imm$
 is locally Lipschitz continuous 
 in the following sense. For any connected open set $\Acal\Subset \HH$ and any $\ep>0$, there exists a constant $C(\ep,\Acal)$ with the following property:  given any $(\gH,\KH)$ and $(\gHbis,\KHbis)$ satisfying  
$$
\aligned 
& \min\big( |\det \gH|, |\det \gHbis| \big) \geq \ep, 
\\
& \max\big(\|\gH\|_{W^{1,p}(\Acal)},\|\KH\|_{L^p(\Acal)},\|\gHbis\|_{W^{1,p}(\Acal)},\|\KHbis\|_{L^p(\Acal)}\big)  \leq \frac1\ep,
\endaligned 
$$
there exists proper isometries $\pi,\wt\pi$ of the Minkowski space such that
$$
\|\wt\pi\circ\immbis-\pi\circ \imm\|_{W^{2,p}(\Acal)} 
\leq C(\ep,\Acal) \,\Big( \|\gHbis-\gH\|_{W^{1,p}(\Acal)} + \|\KHbis-\KH\|_{L^p(\Acal)} \Big). 
$$
\end{theorem}

The paper is organized as follows. In Section~\ref{prel},
we introduce our notation and provide some preliminary results. 
Sections~\ref{manifold}, \ref{hypersurface1}, and \ref{hypersurface2} are devoted to the proof of 
Theorems~\ref{main-one}, \ref{main-two}, and \ref{main-three}, respectively, and contain
slightly more general conclusions.


\section{Notation and preliminaries}
\label{prel}

For background on the analysis techniques (Sobolev spaces on manifolds, etc), we refer to 
\cite{AMR,ADAMS}. 

Throughout this paper, all Greek indices and exponents vary in the set $\{0,1,\ldots,n\}$, 
while Latin indices, save for $n,p,m,\ell$ and $q$, vary in the set $\{1,\ldots,n\}$. 
Einstein summation convention for repeated indices is used. 
A pseudo-Riemannian manifold is a smooth manifold $\MM$ endowed with a metric, that is, 
a symmetric non-degenerate $(0,2)$-tensor field $g$ 
of constant index. A pseudo-Riemannian manifold is called Riemannian if 
its index is zero and Lorentzian if its index is one.  
The Minkowski spacetime $\Mink$ is the vector space $\RR^{n+1}$ endowed with the Minkowski metric 
$$
\gMink(X,Y)=-X^0Y^0 + \sum_{i=1}^n X^i Y^i, \qquad  X=(X^\al), \, Y=(Y^\al) \in \r^{n+1}. 
$$  

An isometric immersion of a pseudo-Riemannian manifold $(\MM,g)$ into another pseudo-Riemannian manifold $(\MM',g')$ is an immersion $\imm:\MM\to \MM'$ that preserves the metric tensor, in the sense that $\imm^* g'=g$. 
Here and in the sequel, notation such as $\imm^* g'$ and $\imm_*X$ denotes respectively the pull-back of $g'$ and the push-forward of $X$ by $\imm$; in particular, $\imm^* g'$ is the $(0,2)$-tensor field on $\MM$ defined by 
$$
(\imm^* g')(X,Y) := g'(\imm_*X,\imm_*Y), \qquad X,Y\in T\MM. 
$$
Related to the Minkowski space, we define the $(n+1)\times (n+1)$ matrix 
\begin{equation}
\label{defA}
\idMink=\diag(-1,1,...,1)
\end{equation} 
and the sets 
$$
\aligned 
 \OO^\eta(n+1)  & :=\{Q\in \r^{(n+1)\times(n+1)}; \ Q^T\idMink Q = \idMink\},\\
 \OO_+^\eta(n+1) & :=\{R\in   \OO^\eta(n+1); \ \det R=1\}. 
\endaligned
$$
The matrices in $\OO^\eta(n+1)$ and $\OO_+^\eta(n+1)$ are respectively called Minkowski-orthogonal and proper Minkowski-orthogonal matrices. By contrast with the set of (usual) orthogonal matrices, i.e., those matrices $P$ that satisfy $P^T P=I$, where $I=\diag(1,...,1)$ denotes the identity matrix of order $n+1$, the set  $\OO^\eta(n+1)$ is not bounded.  An isometry of the Minkowski spacetime $\Mink$ is a mapping
\begin{equation}
\label{isometryMink}
\pi:y\in \r^{n+1}\mapsto v+Qy\in \r^{n+1},
\end{equation}
where $v\in \r^{n+1}$ and $Q\in \OO^\eta(n+1)$. Such an isometry $\pi$ is called proper if $Q\in \OO_+^\eta(n+1)$. 

Let $\Sym(n+1)$ denote the space of all symmetric real matrices. For any matrix $G\in \Sym(n+1)$, let $\lambda_0\leq\lambda_1\leq ...\leq \lambda_n$ denote its (real) eigenvalues. Let 
$$
\LL(n+1):=\{G\in \Sym(n+1); \ \lambda_0<0<\lambda_1\leq ...\leq \lambda_n\}
$$
denote the set of all Lorentz matrices of order $n+1$ and define for all $0<\ep\leq 1$ the subsets
$$
\LL_\ep(n+1) :=\{G\in \LL(n+1); \ |\det G| >\ep \text{ and } |G| <\ep^{-1}\}.
$$
Note that $\LL(n+1)=\lim_{\ep\to 0} \LL_\ep(n+1)$. It is easy to show that any matrix $G\in\LL(n+1)$ has a decomposition $G=F^T \idMink F$ for some invertible matrix $F$ of order $n+1$ (see the beginning of the proof of Lemma \ref{matrixdecomp} below). But such a decomposition is not unique, for the matrix $QF$ with $Q\in  \OO^\eta(n+1)$ also satisfies $G=(QF)^T\idMink (QF)$ (the converse is also true, i.e., if $G=\wt F^T \idMink \wt F$ then $\wt F=QF$ for some $Q\in  \OO^\eta(n+1)$). Since the set $\OO(n+1)$ is not bounded, this shows in particular that the norm of the matrix $F$ in the above decomposition is not controlled by the norm of $G$. This is one of the reasons we need to prove the following lemma about the decomposition of Lorentz matrices: 

\begin{lemma}
\label{matrixdecomp}
Let $G\in \LL_\ep(n+1)$. There exists a mapping $\Fcal:\LL_\ep(n+1)\to \RR^{(n+1)\times(n+1)}$ such that 
$$
\aligned 
& \wt G=\Fcal(\wt G)^T \idMink \Fcal(\wt G), 
\qquad 
|\Fcal(\wt G)|=|\wt G|^{1/2}, 
\\
& 
|\Fcal(\wt G) -\Fcal(G)|\leq C(\ep,n)|\wt G -G|. 
\endaligned 
$$
The mapping $\Fcal$ depends on $G$ but the constant $C(\ep,n)$ does not. 
\end{lemma}

\begin{proof} 
The Euclidean norm of a vector $v\in \r^{n+1}$ is denoted $|v|$, the Euclidean inner product of two vectors $v,w\in \r^{n+1}$ is denoted $v\cdot w$, and the operator norm of a matrice $A\in \r^{(n+1)\times(n+1)}$ is denoted and defined by $|A|:=\sup_{|v|=1}|Av|$. 

Let $p_0,p_1,...,p_n$ be an orthonormal basis in $\r^{n+1}$ formed by eigenvectors of $G$, i.e., $G p_\al =\lambda_\alpha p_\alpha$. Let $P=[p_0\ p_1\ ...\ p_n]$ be the matrix whose $\al$-column is the vector $p_\alpha$. Then $P^T P=I$ and $P^T G P=\diag(\lambda_0,\lambda_1,...,\lambda_n)$. Denoting $A:=\diag((-\lambda_0)^{1/2},\lambda_1^{1/2},...,\lambda_n^{1/2})$ and $F:=AP^T$, we thus have $F=|A|=|G|^{1/2}$
$$
G=P (A \idMink A) P^T=(A P^T)^T \idMink (A P^T)=F^T \idMink F.
$$

Let $\Fcal(G)=F$. To define the matrix $\wt F:=\Fcal(\wt G)$ for any other matrix $\wt G\in \LL_\ep(n+1)$ we distinguish two cases, according to whether $|\wt G-G|< 2 \ep^{n+1}$ or not. This condition is related to the Gram-Schmidt orthonormalisation method. 

Let $\wt\lambda_0<0<\wt\lambda_1\leq ...\leq \lambda_n$ denote the eigenvalues of $\wt G$. Then for every $\al\in\{0,1,...,n\}$, 
$$
|\wt\lambda_\al|\geq \frac{|\det \wt G|}{|\wt G|^n}\geq \ep^{n+1}. 
$$
If $\wt G$ satisfies $|\wt G-G|\geq 2\ep^{n+1}$, we define the matrix $\wt F:=\Fcal(\wt G)$ as follows. Let $\wt p_0, \wt p_1,..., \wt p_n$ be a set of orthonormal eigenvectors of $\wt G$ such that $\wt G \wt p_\al=\wt\lambda_\al \wt p_\al$ and let $\wt P=[\wt p_0\ \wt p_1\ ...\ \wt p_n]$ denote the matrix whose $\al$-column is the vector $\wt p_\alpha$. Then $\wt P^T \wt P=I$ and $\wt P^T \wt G \wt P=\diag(\wt \lambda_0,\wt \lambda_1,...,\wt \lambda_n)$. Denoting $\wt A:=\diag((-\lambda_0)^{1/2},\lambda_1^{1/2},...,\lambda_n^{1/2})$ and $\wt F:=\wt A \wt P^T$, we thus have $|\wt F|=|\wt A|=|\wt G|^{1/2}$ and 
$$
\wt G=\wt P (\wt A \idMink \wt A) \wt P^T=\wt F^T \idMink \wt F
$$
and 
$$
|\wt F -F|\leq |\wt G|^{1/2} + |G|^{1/2} \leq \frac2{\ep^{1/2}}\leq \frac1{\ep^{n+3/2}}|\wt G-G|.
$$

If $\wt G$ satisfies $|\wt G-G|<2\ep^{n+1}$, we define the matrix $\wt F:=\Fcal(\wt G)$ as follows.  Let $\wt p_0\in\r^{n+1}$ be the (unique) unit vector satisfying $\wt G \wt p_0=\wt \lambda_0 \wt p_0$ and $\wt p_0\cdot p_0\geq 0$. Let us first prove the inequalities
\begin{equation}
\label{est1}
|\wt\lambda_0-\lambda_0|\leq |\wt G-G| \ \text{ and } \ 
|\wt p_0 -p_0|\leq \frac1{\ep^{n+1}\sqrt{2}}\, |\wt G-G|.
\end{equation}
Since $\lambda_0=\inf_{|v|=1}\{(Gv)\cdot v\}$, we have
$$
\lambda_0-\wt\lambda_0\leq (G \wt p_0)\cdot \wt p_0 - (\wt G \wt p_0)\cdot \wt p_0 \leq |G-\wt G|. 
$$
By symmetry, this shows that $|\wt\lambda_0-\lambda_0|\leq |\wt G-G|$. 

To prove the second inequality of \eqref{est1}, we decompose the vector $\wt p_0$ as 
$$
\wt p_0=a p_0+b w \ \text{ where } w\in \r^{n+1}, |w|=1, w\cdot p_0=0, w\cdot \wt p_0\geq 0.
$$ 
Clearly, $0\leq a,b\leq 1$ and $a^2+b^2=1$. Then 
$$
\wt\lambda_0 b
=(\wt G \wt p_0)\cdot w
=(G \wt p_0)\cdot w+((\wt G-G) \wt p_0)\cdot w
=b (Gw)\cdot w +((\wt G-G) \wt p_0)\cdot w. 
$$
Since the vector $w$ belongs to the orthogonal complement of $p_0$ in $\r^{n+1}$, it follows that $(Gw)\cdot w\geq \lambda_1$. The last equality above then implies that $(\lambda_1-\wt\lambda_0) b\leq |\wt G-G|$. Since $\lambda_1>\ep^{n+1}$ and $\wt\lambda_0<-\ep^{n+1}$, we have 
$$
b\leq \frac1{2\ep^{n+1}} |\wt G-G|.
$$
Consequently,
$$
|\wt p_0-p_0|^2 = (a-1)^2+b^2\leq (1-a^2)+b^2=2b^2< \frac1{2\ep^{2(n+1)}} |\wt G-G|^2.
$$

We next define an orthogonal basis in $\r^{n+1}$ by applying the Gram-Schmidt orthonormalisation method to vectors $\wt p_0, p_1,...,p_n$. These vectors are linearly independent thanks to the assumption that $|\wt G-G|<2\ep^{n+1}$ and to the relations $2\wt p_0\cdot p_0=|\wt p_0|^2+|p_0|^2-|\wt p_0-p_0|^2\geq 2-\frac1{2\ep^{2(n+1)}} |\wt G-G|^2>0$. We thus define an orthonormal basis $\{\wt v_0, \wt v_1,...,\wt v_n\}$ by letting 
$$
\aligned 
& \wt v_0=\wt p_0 \ \text{ and } \ \wt v_k=\frac1{|v_k|} v_k, \ \text{ where } 
& v_k=p_k-\sum_{i=0}^{k-1}(p_k\cdot \wt v_i)\wt v_i, \ \text{ for all } k=1,...,n.
\endaligned
$$
Clearly, $p_k\cdot v_k\geq 0$ for all $k$. Since $|p_k|=1$, this implies that $|\wt v_k-p_k|\leq \sqrt{2}|v_k-p_k|$. Consequently, for all $k=1,2,...n$, 
$$
\aligned 
|\wt v_k-p_k|^2\leq 2|v_k-p_k|^2=2\sum_{i=0}^{k-1}(p_k\cdot \wt v_i)^2
& = 2\sum_{i=0}^{k-1}(p_k\cdot (\wt v_i-p_i))^2
\\
& \leq 2\sum_{i=0}^{k-1} |\wt v_i-p_i|^2
\endaligned 
$$
and so there exists a constant $C(n)=2^{n-1/2}$ such that 
$$
|\wt v_k-p_k|\leq C(n)|\wt v_0-p_0|\leq \frac{C(n)}{\ep^{n+1}\sqrt{2}} \, |\wt G-G| \ \text{ for all } k=1,...,n. 
$$

Since the subspace of $\r^{n+1}$ spanned by the vectors $\wt v_1,...,\wt v_n$ is stable under the linear mapping defined by the matrix $\wt G$ (because this subspace is the orthogonal complement of the subspace spanned by $\{\wt v_0\}$), $\wt G \wt v_0=\wt\lambda_0 \wt v_0$ and $\wt G \wt v_k=\sum_{i=1}^n \wt H_{ik}\wt v_i$ for some coefficients $\wt H_{ik}\in \r$, $i,k\in \{1,2,...,n\}$. These relations show that the matrices $\wt V=\big[\wt v_0 \ \wt v_1 \ ... \ \wt v_n\big]$ and $\wt H=\big(\wt H_{ik}\big)$ satisfy 
$$
\wt V^T V^T=I 
\text{ \ and \ } 
\wt V^T \wt G \wt V= \mat{\wt\lambda_0}00{\wt H}. 
$$
Note that $\wt H$ is a symmetric and positive-definite matrix whose eigenvalues are precisely $\wt\lambda_1,...,\wt\lambda_n$. Hence $|\wt H|=|\lambda_n|$. Note also that the definition of matrices $\wt V,\wt H$, $P$ and $D=\diag(\lambda_1,...,\lambda_n)$ imply that 
$$
\aligned
& 
|\wt V-P|\leq \Big(\sum_{\al=0}^n |\wt v_\al-p_\al|^2\Big)^{1/2}\leq \frac{C(n)}{\ep^{n+1}} |\wt G-G|,\\
&
|\wt H-D|\leq |\wt V^T \wt G \wt V - P^T G P| \leq |\wt V -P|(|\wt G| + |G|) + |\wt G -G| 
\leq \frac{C(n)}{\ep^{n+2}} |\wt G -G|,
\endaligned
$$
for some constant $C(n)$. Furthermore, since the mapping $A\mapsto A^{1/2}$ is infinitely differentiable on the (convex) set of all symmetric positive-definite matrices, there exists a constant $C(n,\ep)$ (an explicit value is $C(n,\ep)=\frac{\sqrt{n}}{2} \ep^{- \frac{n+1}2}$) such that 
$$
|\wt H^{1/2}-D^{1/2}|\leq C(n,\ep)|\wt H-D|. 
$$

Finally, we define 
$$
\Fcal(\wt G)=\wt F, \text{ \ where } \wt F:=\mat{(-\wt\lambda_0)^{1/2}}00{\wt H^{1/2}}\, \wt V^T. 
$$
This definitions satisfies the conclusions of the theorem, since 
$$
| \wt F|=\max\{|\lambda_0|^{1/2}, |\wt H^{1/2}|\}=|\wt G|^{1/2}
$$
and 
$$
\wt F^T \idMink \wt F =\wt V \ \mat{\wt\lambda_0}00{\wt H}\, V^T=\wt G
$$
and 
$$
\aligned 
& |\wt F -F| =\left| \mat{(-\wt\lambda_0)^{1/2}}00{\wt H^{1/2}}\, \wt V^T 
- \mat{(-\lambda_0)^{1/2}}00{D^{1/2}}\, P^T\right|\\
& 
\leq \max\Big(\left|(-\wt\lambda_0)^{1/2}-(-\lambda_0)^{1/2}\right|, \big|\wt H^{1/2}-D^{1/2}\big|\Big)
+  \max\Big(|\lambda_0|^{1/2}|, \big|D^{1/2}\big|\Big)|\, \wt V^T-P^T|\\
& \leq C(n,\ep)|\wt G-G|
\endaligned
$$
for some constant $C(n,\ep)$ (an explicit value of which is $C(n,\ep)=C(n)\ep^{-\frac{3n+5}2}$). The proof is completed. \end{proof}

The following results on systems of first-order partial differential equations are due to S. Mardare 
\cite[Theorems 1.1 and 4.1]{sor-lpN} and will be used throughout the article.


\begin{theorem}[Existence and uniqueness for Pfaff-type systems]
\label{pfaff}
Let $\Om$ be a connected and simply connected open subset of $\r^m$, let $p>m\ge 2$, 
let $q\geq 1$ and $\ell\geq 1$, and let $x^0\in\Om$ and $Y^0 \in \r^{q\times \ell}$. 
Then, the system of matrix equations
\begin{equation*}
\begin{aligned}
& \frac{\d Y}{\d x^\al} =Y A_\al+B_\al Y + C_\al, 
\\
& Y(x^0)=Y^0, 
\end{aligned}
\end{equation*}
has a unique solution in $W^{1,p}_{\loc}(\Om,\r^{q\times \ell})$ provided
its coefficients $A_\al\in L^p_{\loc}(\Om,\r^{\ell\times\ell})$, $B_\al\in L^p_{\loc}(\Om,\r^{q\times q})$ and $C_\al\in L^p_{\loc}(\Om,\r^{q\times \ell})$ satisfy the compatibility relations 
\begin{equation*}
\begin{aligned}
& \frac{\d A_\be}{\d x^\al}-\frac{\d A_\al}{\d x^\be}=A_\be A_\al - A_\al A_\be &&\text{ in } \D'(\Om,\r^{\ell\times\ell}),\\
& \frac{\d B_\be}{d x^\al} -\frac{\d B_\al}{\d x^\be}=B_\al B_\be - B_\be B_\al &&\text{ in } \D'(\Om,\r^{q\times q}),\\
& \frac{\d C_\be}{\d x^\al}-\frac{\d C_\al}{\d x^\be}=C_\be A_\al - C_\al A_\be + B_\al C_\be - B_\be C_\al   && \text{ in } \D'(\Om,\r^{q\times \ell}).
\end{aligned}
\end{equation*}
\end{theorem}

Although the above result was stated in \cite[Theorem 1.1]{sor-lpN} for systems of the form $\d_\al Y=Y A_\al + C_\al$, it is a simple matter to check that the technique therein extends and provides 
our more general result stated above.  
On the other hand, following closely the proof in \cite[Theorems 4.1 and 6.8]{sor-lpN} we can also check 
the continuous dependence property stated now. 

Let $\Om\subset \r^m$ be a bounded connected open subset with Lipschitz continuous boundary, let $x^0\in\Om$, and let $p>m$ and $\ep>0$. 
Consider any matrices $Y^0,\wt Y^0\in\r^{q\times \ell}$ and matrix fields $A_\al, \wt A_\al \in L^p(\Om,\r^{\ell\times\ell})$, $B_\al, \wt B_\al \in L^p(\Om,\r^{q\times q})$ and $C_\al, \wt C_\al\in L^p(\Om,\r^{q\times \ell})$ 
such that 
$$
\aligned
& 
|Y^0|+\|A_\al\|_{L^p(\Om)}+\|B_\al\|_{L^p(\Om)}+\|C_\al\|_{L^p(\Om)}\leq \ep^{-1},\\
&
|\wt Y^0|+
\|\wt A_\al\|_{L^p(\Om)}+\|\wt B_\al\|_{L^p(\Om)}+\|\wt C_\al\|_{L^p(\Om)}
\leq \ep^{-1}.
\endaligned
$$
If the matrix fields $Y,\wt Y \in W^{1,p}_\loc(\Om,\r^{q\times\ell})$ satisfy the equations
\begin{equation*}
\begin{aligned}
& \frac{\d Y}{\d x^\al} =Y A_\al+B_\al Y + C_\al \text{ \ and \ } Y(x^0)=Y^0, 
\\
&  \frac{\d \wt Y}{\d x^\al} =\wt Y \wt A_\al+\wt B_\al \wt Y + \wt C_\al  \text{ \ and \ }  \wt Y(x^0)=\wt Y^0
\end{aligned}
\end{equation*}
then, for some constant $C_0(\ep)>0$, 
\begin{multline*}
\|Y-\wt Y\|_{W^{1,p}(\Om)} \leq  C_0(\ep) \, \Big\{ |Y(x^0)-\wt Y(x^0)|
\\
+ \sum_\al\left(\|A_\al-\wt A_\al\|_{L^p(\Om)} + \|B_\al-\wt B_\al\|_{L^p(\Om)} + \|C_\al-\wt C_\al\|_{L^p(\Om)}\right) \Big\}.
\end{multline*} 
 

\section{Immersion of a Lorentzian manifold}          
\label{manifold}

Let $(\MM,g)$ be a Lorentzian manifold of dimension $n+1$. 
We want to investigate whether $(\MM,g)$ can be immersed isometrically in the Minkowski space of the same dimension, i.e., whether there exists an immersion $\imm:\MM\to\r^{n+1}$ such that $\imm^* \gMink=g$. 
In coordinates, this condition asserts that, for every local chart $\ph:U\subset \MM \to \Om\subset\r^{n+1}$, 
the composite mapping $f:=\imm\circ \ph^{-1}:\Om\to\r^{n+1}$ satisfies the following two conditions 
$$
\aligned
\det( df) & \neq 0 \, \text{ in } \Om && \text{(that is, $\imm$ is an immersion)},\\
( df)^T \idMink ( df) & = (g_{\al\be})  \, \text{ in } \Om && \text{(that is, $\imm^* \gMink=g$)},
\endaligned
$$
where $g_{\al\be}=g(\d_\al,\d_\be)$ and $\idMink$ is the matrix defined by \eqref{defA}. Here, $\d_\al$ denote the tangent vector fields on $\MM$ along the given coordinates $x^\al$. 
Observe that the matrix $(g_{\al\be})$ is symmetric, invertible, and has exactly one negative 
eigenvalue at every point of $\Om$, since $g$ is assumed to be Lorentzian.

First, we prove 
that the Riemann curvature of a spacetime of dimension $n+1$ must vanish if it is isometrically immersed 
in the Minkowski space of the same dimension.

\begin{lemma}
\label{direct}
If $\imm:(\MM,g)\to \Mink$ is an isometric immersion of class $W^{2,p}_{\loc}$ with $p>n+1$, then the Riemann curvature tensor $\Rm_g$ of $g$ vanishes in the distributional sense over $\MM$. 
\end{lemma}

\begin{proof} 
For smooth immersions, this is a classical result. We need to check that the classical arguments
carry over to an immersion $f$ that is only of class $W^{2,p}_{\loc}$ with $p>n+1$. 
This is in fact the minimal regularity for which the Riemann curvature tensor of $g$ is well-defined as a distribution. 

The assumptions of the lemma show that the mapping $f=\imm\circ\ph^{-1}$ 
belongs to the space $W^{2,p}_{\loc}(\Om,\r^{n+1})$ 
and that the covariant components $g_{\al\be}$ of the metric $g$ satisfy the relations
\begin{equation}
\label{g=g_M}
( df)^T \idMink ( df) = (g_{\al\be}), 
\qquad 
 g_{\al\be}(x)=-\frac{\d f_1}{\d x^\al}\frac{\d f^1}{\d x^\be}
 + \frac{\d f_i}{\d x^\al}\frac{\d f^i}{\d x^\be},
\end{equation}
where $f_\alpha=f^\alpha$ denote the covariant and contravariant components of $f$ with respect to a given Cartesian basis in $\RR^{n+1}$. Since $\displaystyle \frac{\d f_\al}{\d x^\be}\in W^{1,p}_{\loc}(\Om)$ and this space is in fact an algebra (we use here the assumption $p>n+1$), the above relation implies that $g_{\al\be}\in W^{1,p}_{\loc}(\Om)$. 
In view of the definitions of the inverse of a matrix and of the Christoffel symbols, this implies that 
$$
\aligned
& (g^{\si\nu})=(g_{\al\be})^{-1}\in W^{1,p}_{\loc}(\Om),\\ 
& \Gamma^\si_{\al\be} := \frac12 g^{\si\nu} \Big(\frac{\partial}{\partial x^\al}  g_{\be\nu}+ 
\frac{\partial}{\partial x^\be} g_{\nu \al} - \frac{\partial}{\partial x^\nu} g_{\al\be}\Big) \in L^p_{\loc}(\Om).
\endaligned
$$ 
Hence, the Riemann curvature tensor of the metric $g$, defined by 
$$
R^\ta_{\ \si\al\be} := \frac{\partial}{\partial x^\al} \Gamma^\ta_{\be\si} - \frac{\partial}{\partial x^\be} \Gamma^\ta_{\al\si} + 
\Gamma^\nu_{\be\si} \Gamma^\ta_{\al\nu} - \Gamma^\nu_{\al\si} \Gamma^\ta_{\be\nu},
$$
is well defined as the sum of a distribution in $W_\loc^{-1,p}(\Om)$ and a function $L^{p/2}_\loc(\Om)$. Here, the space $W_\loc^{-1,p}(\Om)$ is defined by 
$$
W_\loc^{-1,p}(\Om):=\Big\{u\in\Dcal'(\Om); \ u=f+\sum_{\al=0}^n \frac{\d f^\alpha}{\d x^\alpha}, \text{ for some } f,f^\alpha\in L^p_\loc(\Om)\Big\}.
$$

We will now show that $R^\ta_{\ \si\al\be}=0$ in the distributional sense. 
Define $\displaystyle F_\si(x):=\frac{\d f}{\d x^\si}(x)$, $\si=0,1,\ldots,n$. 
Since these vectors form a basis in $\r^{n+1}$ for every $x\in\Om$, 
the vector $\frac{\d F_\si}{\d x^\al}(x)$ can be decomposed over this basis
and so there exist coefficients $C^\ta_{\al\si}$ such that 
\begin{equation}
\label{connect}
\frac{\d F_\si}{\d x^\al}=C_{\al\si}^\be F_\be \quad \text{ in } \Om.
\end{equation}
Since $\frac{\d F_\si}{\d x^\al}=\frac{\d F_\al}{\d x^\si}$ and $(F_\si)^T \idMink F_\ta=g_{\si\ta}$ (see \eqref{g=g_M}), the coefficients $C^\be_{\al\si}$ must satisfy the relations 
$$
\aligned
C^\be_{\al\si}&=C^\be_{\si\al},\\
\frac{\del g_{\si\ta}}{\d x^\al}&=C_{\al\si}^\be g_{\be\ta}+g_{\si\be}C_{\al\ta}^\be.
\endaligned
$$ 
Solving this system shows that 
$$
C^\be_{\al\si}
=\frac12 g^{\be\nu} (\frac{\partial}{\partial x^\al} g_{\nu\si} + \frac{\partial}{\partial x^\si} g_{\al\nu}-\frac{\partial}{\partial x^\nu} g_{\al\si}) 
=\Ga^\be_{\al\si}.
$$
Hence, the partial derivatives $F_\si=\frac{\d f}{\d x^\si}$ of $f$ actually 
satisfy the equations 
$$
\frac{\d F_\si}{\d x^\al}=\Ga_{\al\si}^\be F_\be,
$$
which can be rewritten as matrix equations, that is,  
\begin{equation}
\label{connect.bis}
\frac{\d}{\d x^\al}(df) = ( df)\Ga_\al, 
\end{equation}
where $\Ga_\al:=(\Ga_{\al\si}^\be)$ denote the matrix field with $\Ga_{\al\si}^\be$ as its component at the $\be$-row and $\si$-column. In particular, we see that the mapping $\imm$ preserves the connection. 

To conclude, we now rely on the commutativity property for second derivatives of the fields $F_\al$ 
and obtain: 
$$
 \frac{\d}{\d x^\be}\left(\Ga_{\al\si}^\ta F_\ta\right)
=\frac{\d}{\d x^\al}\left(\Ga_{\be\si}^\ta F_\ta\right). 
$$
Combined with the relation $\frac{\d F_\si}{\d x^\al}=\Ga_{\al\si}^\be F_\be$, this 
implies that, for every test function $v\in \Dcal(\Om)$, 
$$
\int_\Om\Big(-\Ga_{\al\si}^\ta\frac{\d}{\d x^\be}(F_\ta v)
+ \Ga_{\al\si}^\ta \Ga_{\be\ta}^\nu (F_\nu v)  \Big)\, dx
=
\int_\Om\Big(-\Ga_{\be\si}^\ta\frac{\d}{\d x^\al}(F_\ta v)
+ \Ga_{\be\si}^\ta \Ga_{\al\ta}^\nu (F_\nu v) \Big)\, dx.
$$
Since the vectors fields $F_0,F_1,...,F_n$ form a basis in $W^{1,p}_\loc(\Om;\r^{n+1})$, we can define the dual basis $F^0,F^1,...,F^n$ and use the components of the vector field $F^\ga w$, where $w$ is any test function in $\Dcal(\Om)$, in lieu of $v$ in the above equations. This implies that 
$$
\int_\Om\Big(-\Ga_{\al\si}^\ga\frac{\d}{\d x^\be} w
+ \Ga_{\al\si}^\ta \Ga_{\be\ta}^\ga w  \Big)\, dx
=
\int_\Om\Big(-\Ga_{\be\si}^\ga\frac{\d}{\d x^\al} w
+ \Ga_{\be\si}^\ta \Ga_{\al\ta}^\ga w \Big)\, dx. 
$$
Since the test function $w$ was arbitrary, the above equation is equivalent with the following equation between distributions in $\Dcal'(\Om)$: 
$$
R^\ta_{\  \si \al\be} = 0. 
$$
\end{proof}

The next lemma establishes a partial converse to Lemma~\ref{direct}, that is,
 if the Riemann curvature of the metric $g$ vanishes, then the Lorentzian manifold $(\MM,g)$ can be \emph{locally} isometrically immersed in the Minkowski space of the same dimension. We recall that an isometry $\pi:\Mink\to\Mink$ of the Minkowski spacetime $\Mink$ was defined in \eqref{isometryMink}.

\begin{lemma}
\label{converse.local} 
Suppose that $g$ is of class $W^{1,p}_\loc(\MM)$ with $p>n+1$ and that $\Rm_g=0$ in the distributional sense. 
Consider a connected and simply 
connected open subset $U\subset \MM$ that can de described by a single local chart. 
Then, there exists an immersion $\imm:U\to\r^{n+1}$ of class $W^{2,p}_\loc(U)$ such that $\imm^* \gMink = g$. Moreover, if $\pi$ is an isometry of the Minkowski spacetime $\Mink$, then $\pi\circ\imm$ also satisfy $(\pi\circ\imm)^* \gMink = g$.
\end{lemma}

\begin{proof}
Let $\ph:U\subset \MM \to \Om\subset \r^{n+1}$ be a
 local chart such that $\Om:=\ph(U)$ is connected and simply 
connected. The components $g_{\al\be}$ of the metric with respect to this chart
belongs to $W^{1,p}_\loc(\Om)$. This space being an algebra, the inverse of matrix field $(g_{\al\be})$ 
and the Christoffel symbols $\Gamma^\si_{\al\be}$ satisfy 
$$
\aligned 
& (g^{\si\ta}):=(g_{\al\be})^{-1}\in W^{1,p}_\loc(\Om),\\
& \Gamma^\si_{\al\be} := \frac12 g^{\si\nu} (\frac{\partial}{\partial x^\al} g_{\be\nu} + 
\frac{\partial}{\partial x^\be} g_{\nu \al} - \frac{\partial}{\partial x^\nu} g_{\al\be})\in L^p_\loc(\Om).
\endaligned
$$ 
Finding an isometric immersion $\imm:(U,g) \to (\r^{n+1},\gMink)$ is equivalent to finding a mapping $f:\Om \to \r^{n+1}$ such that 
\begin{equation}
\label{eq2}
( df)^T \idMink ( df)=(g_{\si\ta}) \text{ in } \Om,
\end{equation}
where $df$ denotes the gradient of $f$ (the set $\Om$ and the vectorial space $\r^{n+1}$ are equipped with the usual Cartesian coordinates and bases). 
Thus we are left with solving a nonlinear matrix equation in the unknown $f$. We are going to show
that the equation \eqref{eq2} can be solved if the components of the matrix $(g_{\si\ta})$ are the covariant components of a metric whose Riemann curvature tensor vanishes in the distributional sense. 

We first observe that the equation \eqref{eq2} can be solved ``at one point'' $x_\star\in\Om$, since by Lemma \ref{matrixdecomp} there exists a matrix $F_\star\in \r^{(n+1)\times (n+1)}$ that satisfies
$$
F_\star^T \idMink F_\star = (g_{\si\ta}(x_\star)).
$$

Solving the nonlinear 
equation \eqref{eq2} is now reduced to solving two \emph{linear} systems. The first one is the Pfaff system 
\begin{equation}\label{1}
\aligned
\frac{\d F}{\d x^\al}& = F \,\Gamma_\al \quad \text{ a.e. in } \Om,\\
F(x_\star)& =F_\star,
\endaligned
\end{equation}
where, for every $\al$, $\Ga_\al:=(\Ga_{\al\be}^\si)$ denotes the matrix field with $\Ga_{\al\be}^\si$ as its element at the $\si$-row and $\be$-column. This system has a solution $F\in W^{1,p}_\loc(\Om;\r^{(n+1)\times (n+1)})$ by Theorem~\ref{pfaff}, since the compatibility conditions 
$$
\frac{\d \Ga_\be}{\d x^\al}-\frac{\d \Ga_\al}{\d x^\be}=\Ga_\be \Ga_\al - \Ga_\al \Ga_\be 
\qquad \text{ in } \D'(\Om;\r^{(n+1)\times (n+1)})
$$
are equivalent to the fact that the Riemann curvature tensor of $g$ vanishes in the distributional sense.  

The second system is the Poincar\'e system 
\begin{equation}
\label{1bis}
df = F  \quad \text{ in } \Om, 
\end{equation}
where $F$ is the solution of the first system \eqref{1}. The above Poincar\'e system has a solution $f\in W^{2,p}_\loc(\Om;\r^{n+1})$ since  the compatibility conditions that $F$ must satisfy, namely 
$$
\frac{\d F_\be}{\d x^\al}= \frac{\d F_\al}{\d x^\be},  \quad \text{ where } F_\al  \text{ denotes the $\al$-column of the matrix field } F, 
$$
are equivalent to relations $\Ga_{\al\be}^\si=\Ga_{\be\al}^\si$ in view of \eqref{1}. Or these relations are indeed satisfied by the Christoffel symbols. 

It remains now to prove that the solution $f$ of the system \eqref{1bis} satisfies the equation $( df)^T \idMink ( df)=(g_{\si\ta})$. 
Combining relations \eqref{1} and \eqref{1bis} shows that $f$ satisfies the equation $\displaystyle \frac{\d ( df)}{\d x^\al}=( df)\Ga_\al$. This implies that the matrix field $( df)^T \idMink ( df)$ satisfies the equation
$$
\frac{\d}{\d x^\al}\Big[ ( df)^T \idMink ( df)\Big]
= \Ga_\al^T \Big[ ( df)^T \idMink ( df)\Big]+ \Big[( df)^T \idMink ( df)\Big] \Ga_\al. 
$$
But the matrix field $(g_{\si\ta})$ also satisfies this equation, i.e., 
$$
\frac{\d}{\d x^\al}(g_{\si\ta}) = \Ga_\al^T (g_{\si\ta}) + (g_{\si\ta})\Ga_\al, 
$$
since 
this is simply a rewriting of the definition of the Christoffel symbols. In addition, the definition of the matrix $F_\star$ 
shows that $\Big[( df)^T \idMink ( df)\Big](x_\star)=(g_{\si\ta})(x_\star)$. Therefore the uniqueness part of Theorem~\ref{pfaff} shows that the two solutions coincide, i.e., $( df)^T \idMink ( df)=(g_{\si\ta})$ in the entire domain $\Om$. 

Finally, if $\pi$ is any isometry of the Minkowski spacetime $(\r^{n+1},\gMink)$, it is easily checked that $\pi\circ\imm$ is also an isometric immersion of $(U,g)$ into $(\r^{n+1},\gMink)$. 
 \end{proof}

Lemma \ref{converse.local} is a local existence result, in the sense that the isometric immersion $\imm$ is defined only on a subset $U\subset \MM$. But we have a good control of the size  of this neighborhood, since $U$ only needs to be connected, simply 
connected, and defined by a single local chart. This fact, together with the uniqueness result of Lemma \ref{uniqueness} below, allow us to establish a global existence result (i.e., the isometric immersion $\imm$ is defined over the entire manifold $\MM$) when $\MM$ is simply connected (see Lemma~\ref{converse.global} below). 

We next prove a stability result for the isometric immersion $\imm$ defined by Lemma \ref{converse.local}.

\begin{lemma}
\label{stability} 
For each $\ep>0$ and connected smooth open set $\Acal\Subset \MM$, there exists a constant $C=C(\ep,\Acal)$ with the following property: If 
the metrics $g=\imm^*\gMink$ and $\gbis={\immbis}^*\gMink$ induced by the immersions $\imm,\immbis:\MM\to\Mink$ of class $W^{2,p}_\loc$ with $p>n+1$ satisfy  
$$
\min\big( |\det g|, |\det\gbis| \big) \geq \ep, 
\qquad 
\max\big(\|g\|_{W^{1,p}(\Acal)},\|\gbis\|_{W^{1,p}(\Acal)}\big) \leq \frac{1}{\ep}\ ,
$$
then there exists isometries $\pi$ and $\wt\pi$ of the Minkowski space $\Mink$ such that 
$$
\|\wt\pi\circ\immbis-\pi\circ\imm\|_{W^{2,p}(\Acal)}\leq C \, \|\gbis- g\|_{W^{1,p}(\Acal)}.
$$
\end{lemma}

\begin{remark} 
(i) The constant $C(\ep,\Acal)$ may go to infinity when either $\ep\to 0$ or $\Acal$ ``approaches'' the whole manifold $\MM$. 

(ii) The inequality above shows that small perturbations of the metric $g$ induce small perturbations of the immersion $\imm$, defined up to an isometry of the Minkowski space. 
\end{remark}

\begin{proof}
Let $\ph:U\subset \MM\to \Om\subset\r^{n+1}$ be a local chart such that $U$ is connected with a smooth boundary. Let $f:=\imm\circ\ph^{-1}$ and $\wt f:=\immbis\circ\ph^{-1}$.  
In coordinates, the relations $g=\imm^*\gMink$ and $\gbis={\immbis}^*\gMink$ read 
\begin{equation}
\label{eq3}
(df)^T\idMink( df)=(g_{\si\ta}), 
\qquad  
(d\wt f)^T\idMink(d\wt f)=(\gbis_{\si\ta}) \quad \text{ in } \Om.
\end{equation}
As in the proof of Lemma~\ref{direct} (see \eqref{connect.bis}), these relations imply that $f,\wt f\in W^{2,p}_\loc(\Om;\r^{n+1})$  and that the matrix fields $ df$ and $d\wt f$ satisfy the equations 
$$
\frac{\d}{\d x^\al}( df)=( df) \Ga_\al, 
\qquad 
\frac{\d}{\d x^\al}(d\wt f)=(d\wt f) \wt \Ga_\al \qquad \text{ in } \Om,
$$
where $\Ga_\al$ is the matrix field defined from the metric $g$ as in the proof of Lemma~\ref{converse.local} and $\wt\Ga_\al$ is defined in the same way but with $g$ replaced by $\gbis$. 

Fix any point $x_\star\in\Om$. By Lemma~\ref{matrixdecomp}, there exist matrices $F$ and $\wt F$ such that 
\begin{equation}
\label{atxstar}
\aligned
& (g_{\si\ta}(x_\star))=F^T \idMink F \ \text{ and } \ 
(\gbis_{\si\ta}(x_\star))=\wt F^T \idMink \wt F, \\
& |F|=\big| (g_{\si\ta}(x_\star))\big|^{1/2} \text{ \ and \ } |\wt F|=\big| (\wt g_{\si\ta}(x_\star))\big|^{1/2}\\
& |\wt F- F| \leq C \, |(g_{\si\ta}(x_\star))-(\gbis_{\si\ta}(x_\star))| \leq C\, \|(g_{\si\ta})-(\gbis_{\si\ta})\|_{W^{1,p}(\Om)}, 
\endaligned
\end{equation}
for some constant $C$ depending on $\ep,n$. 
Combined with Equations \eqref{eq3} (applied at the point $x_\star$), the first relations above imply that the matrices $Q=F(df(x_\star))^{-1}$ and $\wt Q=\wt F (d\wt f(x_\star))^{-1}$ (the matrices $(df(x_\star))$ and $(d\wt f(x_\star))$ are invertible since $\imm$ and $\immbis$ are immersions) satisfy the relations 
$$
Q^T \idMink Q =\idMink, \qquad \wt Q^T \idMink \wt Q =\idMink.
$$
Since the matrix fields $(Q \, df)$ and $(\wt Q \,d\wt f)$ satisfy the equations 
$$
\aligned
\frac{\d}{\d x^\al}(Q \, df)=(Q \, df) \, \Ga_\al,
& \quad
\frac{\d}{\d x^\al}(\wt Q \, d\wt f))=(\wt Q \, d\wt f) \,\wt \Ga_\al 
\quad \text{ in } \Om,\\
(Q \, df)(x_\star)= F,
& \qquad  
(\wt Q \,d\wt f)(x_\star) = \wt F,
\endaligned
$$
the stability property for Pfaff's systems stated in Section~\ref{prel} implies that 
$$
\|Q \, df - \wt Q \, d\wt f \|_{W^{1,p}(\Om)} \leq C \, 
\Big( |F- \wt F|
+ \sum_\al\|\Ga_\al-\wt\Ga_\al\|_{L^p(\Om)}\Big),
$$
where the constant $C$ depends only on $n, \ep$ and $\Om$. 

On the other hand, the definition of the Christoffel symbols $\Ga_{\al\be}^\si$ and $\wt\Ga_{\al\be}^\si$ shows that
$$
\sum_\al\|\Ga_\al-\wt\Ga_\al\|_{L^p(\Om)}
\leq C \, \|(g_{\si\ta})-(\gbis_{\si\ta})\|_{W^{1,p}(\Om)}.
$$
Usig this inequality and the second inequality of \eqref{atxstar} in the previous inequality yields 
$$
\|Q \, df-\wt Q \, d\wt f \|_{W^{1,p}(\Om)} \leq C \, \|(g_{\si\ta})-(\gbis_{\si\ta})\|_{W^{1,p}(\Om)}.
$$
This inequality in turn implies (thanks to Poincar\'e-Wirtinger's inequality) that 
$$
\|(v+Q f) - (\wt v+\wt Q \wt f)\|_{W^{2,p}(\Om)} \leq C \, 
\|(g_{\si\ta})-(\gbis_{\si\ta})\|_{W^{1,p}(\Om)}, 
$$
where $v=-Qf(x_\star)$ and $\wt v=- \wt Q \wt f(x_\star)$. Since the matrices $Q$ and $\wt Q$ are Minkowski-orthogonal, the mappings $\pi: y\in\r^{n+1}\mapsto v+ Qy\in\r^{n+1}$ and  $\wt \pi: y\in\r^{n+1}\mapsto \wt v+ \wt Qy\in\r^{n+1}$ are isometries of the Minkowski spacetime $\Mink$. Letting $x=\ph(p)$, $p\in U$, in the above inequality shows that
$$
\|\wt\pi\circ \immbis-\pi\circ\psi\|_{W^{2,p}(U)} \leq C \, \|\gbis-g\|_{W^{1,p}(U)}. 
$$
This inequality still holds when $U$ is replaced with the possibly larger set $\Acal$ since $\Acal$ is connected and $\overline{\Acal}$ is compact.
\end{proof}

An immediate consequence of the previous lemma is the following uniqueness result. 

\begin{lemma}
\label{uniqueness} 
If $\imm,\immbis:(\MM,g)\to (\r^{n+1},\gMink)$ are isometric immersions of class $W^{2,p}_{loc}$ with $p>n+1$, then for every connected component of $\MM$ there exists an isometry $\ta$ of the Minkowski spacetime $(\r^{n+1},\gMink)$ such that $\immbis=\ta\circ\imm$. 
\end{lemma}

Now, we are in a position to establish a \emph{global} version of Lemma \ref{converse.local}. This completes the proof of Theorem~\ref{main-one} stated in the introduction.

\begin{lemma}\label{converse.global}
Suppose that $\MM$ is simply 
connected, $g$ is of class $W^{1,p}_\loc(\MM)$ with $p>n+1$, and $\Rm_g=0$. Then,
 there exists an isometric immersion $\imm:(\MM,g)\to(\r^{n+1},\gMink)$ of class $W^{2,p}_\loc(\MM)$. 
\end{lemma}

\begin{proof}
We can assume without losing in generality that the manifold $\MM$ is connected, otherwise it suffices to apply the 
forthcoming
 argument in each connected component of $\MM$. The idea is to patch together sequences of local isometric immersions $\imm_m:U_m\subset \MM \to \r^{n+1}$ constructed
  in Lemma \ref{converse.local}. The uniqueness result of Lemma \ref{uniqueness} allows us to choose $\imm_m$ in such a way that $\imm_i=\imm_j$ on the overlapping domain $U_i\cap U_j$. The simple-connectedness of $\MM$ insures that this definition is unambiguous (i.e., it does not depend on the choice of the local isometric immersions).  

To begin with, we fix a point $p_0\in \MM$ and a local chart $(\ph_0,U_0)$ at $p_0$, where $U_0\subset \MM$ is a connected and simply 
connected neighborhood of $p_0$ and $\ph_0:U_0\subset \MM\to \ph_0(U_0)\subset \r^{n+1}$. Then,
 Lemma \ref{converse.local} shows that there exists an isometric immersion $\imm_0:(U_0,g)\to (\r^{n+1},\gMink)$. 

Given any $p\in \MM$, we choose a path $\ga:[0,1]\to \MM$ joining $p_0$ to $p$ (i.e., a continuous function $\ga:[0,1]\to \MM$ with $\ga(0)=p_0$ and $\ga(1)=p$). Next we construct a division $\De:=\{t_0,t_1,t_2,...,t_K,t_{K+1}\}$ of the interval $[0,1]$, with $K\in \n$ and $0=t_0<t_1<...<t_K<t_{K+1}=1$, and a sequence of local charts $(\ph_m, U_m)_{m=1}^K$, with $\ga(t_m)\in U_m$ and $\ph_m:U_m\subset \MM \to \ph_m(U_m)\subset\r^{n+1}$, in such a way that $U_m$ is connected and simply 
connected and $\ga([t_{m},t_{m+1}])\subset U_{m}$ for all $m=0,1,2,...,K$. Let us prove that such a construction is possible. 

Consider the set $\mathcal A$ of all pairs $\big(\De, (\ph_m,
U_m)_{m=1}^K\big)$, with $K=|\De|-2$ (where $|\De|$ denotes the
cardinality of the set $\De$), that satisfies all the properties of the
above paragraph with the only difference that now $t_{K+1}\le 1$ only
(in other words, the last element of $\De$ need not be equal to
$1$). It is enough to prove that
$$
\sup_{\mathcal A} t_{K+1}=1=\max_{\mathcal A} t_{K+1}. 
$$
Obviously,
$\mathcal A$ is not empty and $\sup_{\mathcal A} t_{K+1}>0$.

Assume on the contrary that $s:=\sup_{\mathcal A} t_{K+1}<1$ and consider a local chart $(\ph,U)$ at $\gamma(s)$ such that $U$ is connected and simply 
connected. Since $\gamma$ is continuous and $U$ is a neighbourhood of $\gamma(s)$, there exists $0<\ep<\min\{s,1-s\}$ such that $\gamma([s-\ep,s+\ep])\subset U$. On the other hand, since $s$ is a supremum, there exists a pair $\big(\Delta, (\ph_m, U_m)_{m=1}^{L}\big)$, where $L=|\De|-2$, such that $t_{L+1}\ge s-\ep$ ($t_{L+1}$ being the last element of the set $\De$). It is easy to check that the pair $\big(\Delta',(\ph_m, U_m)_{m=1}^{L+1}\big)$, where 
$$
\De':=\De\cup \{t_{L+2}:=s+\ep\}, \qquad 
(\ph_{L+1}, U_{L+1}):=(\ph,U),
$$
 belongs to $\mathcal A$. Hence,
 $\sup_{\mathcal A} t_{K+1} \ge s+\ep$, which contradicts the definition of $s$. Therefore, $s=1$. 

In order to prove that $s=1$ is in fact a maximum (i.e., $\sup_{\mathcal A} t_{K+1}=\max_{\mathcal A} t_{K+1}$), we consider a local chart $(\ph,U)$ at $p=\gamma(1)$ and we repeat the argument above with the only difference that now $\ep$ is chosen such that $0<\ep<1$ and $\gamma([1-\ep,1])\subset U$.

With $\ga$, $\De$, and $(\ph_m, U_m)_{m=1}^K$ constructed as above, Lemmas \ref{converse.local} and \ref{uniqueness} allow us to successively choose isometric immersions 
$$
\imm_1:(U_1,g)\to (\r^{n+1},\gMink), ..., \imm_K:(U_K,g)\to (\r^{n+1},\gMink),
$$
 in such a way that $\imm_m=\imm_{m-1}$ on the connected component containing $\ga(t_m)$ of $U_m\cap U_{m-1}$, for all $m=1,2,...,K$. 

Finally, we define a mapping $\imm:(\MM,g)\to (\r^{n+1},\gMink)$ by $\imm(p):=\imm_K(p)$, where $\imm_K$ is defined as above. Indeed, one can see that this definition is independent on the choice of $\imm_K$ by using the simple-connectedness of the manifold $\MM$ (a similar argument was used in \cite{sor-lpN}). 
 Then, $\imm$ is clearly an immersion since this property is local.
\end{proof}


\section{Immersion of a hypersurface with general signature} 
\label{hypersurface1}

We now turn our attention to hypersurfaces $\HH\subset \MM$ in a Lorentzian manifold $\MM$ of dimension $n+1$. The basic question we addressed here is whether such a hypersurface can be immersed in the Minkowski spacetime $(\r^{n+1},\gMink)$  by means of an immersion $\imm:\HH\to\r^{n+1}$ that preserves the geometry of the hypersurface. 
 
To begin with, we recall the corresponding results in Riemannian geometry. Let $\HH\subset \MM$ be a hypersurface in a Riemannian manifold $(\MM,g)$ of dimension $n+1$. Then, $\HH$ is endowed with first and second fundamental forms, which together characterize the geometry of the hypersurface. The first fundamental form $\gH$ is the pull-back of $g$ on $\HH$, while the second fundamental form $\KH$ is the pull-back of $\nabla \nH$ on $\HH$, where $\nH$ is a normal form to the hypersurface $\HH$ and $\nabla$ is the Levi-Civita connection induced by $g$ on $\MM$. Then, the question is whether there exists an immersion $\imm:\HH\to \e^{n+1}$, where $\e^{n+1}$ denotes the Euclidean space of dimension $n+1$, such that $\gH$ and $\KH$ are the first and second fundamental forms induced by $\imm$. In the classical setting (i.e., all data are smooth), Bonnet's theorem asserts that such an immersion exists locally if and only if the fundamental forms satisfy the Gauss and Codazzi equations. Subsequently, this theorem was generalized by Hartman \& Wintner \cite{hawi} to fundamental forms $(\gH, \KH)$ that are only in $\C^1\times\C^0$, 
and finally by S. Mardare \cite{sor-lp,sor-lpN} to the case 
$(\gH,\KH)\in W^{1,p}\times L^p$ with $p>n$, the latter regularity being optimal. 

If $(\MM,g)$ is now a Lorentzian manifold, the first fundamental form of a hypersurface $\HH\subset \MM$ 
need not provide useful information about the geometry of $\HH$: If the hypersurface is null, then its normal vector field is a null vector lying in the tangent bundle to the hypersurface and the first fundamental form is degenerate. For this reason, the first fundamental form is replaced in the Lorentzian setting with a connection $\naH$ on $\HH$, defined by projecting the Levi-Civita connection $\nabla$ (associated with the Lorentzian metric $g$)  along a prescribed 
vector field $\rig$ transversal to $\HH$. Such a vector field $\rig\in T\MM$ is called a {\sl rigging}
 and must satisfy 
$$
\rig_p\not\in T_p \HH, \qquad p\in \HH. 
$$
It is convenient to normalize $\rig$ by imposing $\langle \nH_p, \rig_p\rangle =1$ for all $p\in \HH$, where $\nH$ denotes as usual a normal form on $\HH$ chosen once and for all.
For a mathematical presentation of the notion of rigging, we refer to LeFloch and C. Mardare \cite{lfma}.

Given a hypersurface $\HH\subset \MM$ and a rigging $\rig\in T\MM$, the 
{\sl rigging projection on $\HH$} is denoted $X\in T\MM \mapsto \ov X \in T\HH$ and is defined by setting 
$$
X=\ov X + \la \nH,X\ra \rig. 
$$
Then the connection $\naH$ on $\HH$ is defined by 
$$
\naH_X Y =\ov{\nabla_X Y}, \quad  X,Y\in T\HH,
$$
and the second fundamental form $\KH$ of $\HH$ is defined by 
$$
\KH(X,Y)=\la \nabla_X \nH,Y\ra, \quad  X,Y\in T\HH.
$$

Beside the connection $\naH$ and the second fundamental form $\KH$ which in a sense characterize the geometry of 
the hypersurface $\HH$, we must introduce 
additional operators characterizing the rigging vector $\rig$. Guided by the decomposition
\begin{equation}
\label{l=...}
\nabla_X \rig=\ov{\nabla_X \rig} +\la \nH, \nabla_X \rig\ra \rig=\ov{\nabla_X \rig} -\la \nabla_X \nH, \rig\ra \rig, \quad  X\in T\HH, 
\end{equation}
we define the operators $\LH:T\HH\to T\HH$ and $\MH:T\HH\to\r$ by 
$$
\LH(X)=\ov{\nabla_X \rig}, \qquad  \MH(X)=\la \nabla_X \nH, \rig \ra, \quad X\in T\HH.
$$
We note that the operators $\naH,\KH, \LH, \MH$ can be also introduced via the decompositions: 
\begin{equation}
\label{decomp1}
\aligned
\nabla_X Y  & =\naH_X Y-\KH(X,Y)\rig, \quad  X,Y\in T\HH,\\
\nabla_X \rig & = \LH(X) - \MH(X)\rig, \quad  X\in T\HH.
\endaligned
\end{equation}

\begin{remark}
\label{rk2}
Define the operator $\mbox{ }^\sharp:\theta\in T^*\MM\to \theta^\sharp\in T\MM$ by 
$$
\la \theta, Y\ra=g(\theta^\sharp, Y), \quad Y\in T\MM.
$$ 
If the vector field $\rig=\nH^\sharp$ is transversal to $\HH$, then $\naH$ is the Levi-Civita connection induced by the first fundamental form of $\HH$,  
$\LH(X)=\KH(X,\cdot)^\sharp$, and $\MH=0$. Thus the operators $\naH,\KH, \LH, \MH$ associated to a pair $(\HH,\rig)$ are defined in terms of the fundamental forms associated with the pair $(\HH,\nH^\sharp)$; cf. also 
Section \ref{hypersurface2} below.
\end{remark}

Our principal objective in this section is to prove that the operators $\naH, \KH, \LH, \MH$ characterize the 
pair $(\HH,\rig)$ 
formed by the hypersurface and the rigging vector field. 
We are going to generalize Bonnet's theorem in the Lorentzian setting to 
a pair $(\HH,\rig)$. Let $\Mink=(\r^{n+1},\gMink)$ be the Minkowski spacetime of dimension $n+1$ and let $\naMink$ be the Levi-Civita connection induced by $\gMink$. If $\imm:\HH\to \r^{n+1}$ is an immersion and $\rigMink:\HH\to T\r^{n+1}$ is a rigging along $\HH':=\imm(\HH)$ (this definition makes sense since $\HH'$ is locally a hypersurface of $\r^{n+1}$), then we define
 the operators $\naHMink, \KHMink, \LHMink, \MHMink$ on $\HH'$ via the decompositions (similar to \eqref{decomp1}):
\begin{equation}
\label{decomp2}
\aligned
\naMink_{X} Y   & =\naHMink_{X} Y - \KHMink(X,Y)\rigMink,\quad X,Y\in T\HH',
\\
\naMink_{X} \rigMink & =\LHMink(X)  - \MHMink(X)\rigMink,\quad X\in T\HH'.
\endaligned
\end{equation}

We say that an immersion $\imm:\HH\to\r^{n+1}$ and a rigging $\rigMink:\HH\to T\r^{n+1}$ along
 $\HH'=\imm(\HH)$ {\sl preserve the operators}  $\naH,\KH, \LH, \MH$ if, for all $X,Y\in T\HH$,
$$
\aligned
& \imm_*(\naH_X Y)=\naHMink_{\imm_*X}\imm_* Y, \quad \imm_*(\LH(X))=\LHMink(\imm_* X),\\
& \KH(X,Y)=\KHMink(\imm_*X,\imm_* Y), \quad \MH(X)=\MHMink(\imm_* X).
\endaligned
$$
Equivalently, this means that 
$$
\aligned
\naMink_{\imm_*X} \imm_*Y   & =\imm_*(\naH_{X} Y) - \KH(X,Y)\rigMink, \quad X,Y\in T\HH,
\\
\naMink_{\imm_*X} \rigMink & =\imm_*(\LH(X))  - \MH(X)\rigMink, \quad X\in T\HH.
\endaligned
$$
Throughout this section, we assume that $\imm\in W^{2,p}_\loc(\HH)$ and $\rig\in W^{1,p}_\loc(\HH)$ with $p>n$; the corresponding operators $\naH,\KH, \LH, \MH$ belong to the space $L^p_\loc(\HH)$. This regularity is sharp in the sense that the forthcoming 
results do not hold under a lower regularity. 

We will show that such an immersion exists locally if and only if the operators $\naH,\KH, \LH, \MH$ satisfy  generalized Gauss and Codazzi equations, defined as follows. Let $\ph:U\subset \HH\to \Om\subset \r^{n}$ be a local chart, let $x^i$, $i=1,2,...,n$, be a set of Cartesian coordinates in $\Om$, and let $\d_i$ be the vector field tangent to the coordinate line $x^i$. Then, $\{\d_1,...,\d_n\}$ form a basis of the tangent space
 $T\HH$ to the hypersurface $\HH\subset \MM$, while $\{\rig,\d_1,...,\d_n\}$ form a basis of the tangent 
 space to $\MM$ thanks to the definition of the rigging $\rig$. We denote $R_{\ hij}^k$ the components of the Riemann curvature tensor field associated with the connection $\naH$ and $K_{ij}$, $L_j^k$, $M_i$ the components 
 of the operators $\KH$, $\LH$, $\MH$, respectively. 

\begin{lemma}
\label{direct+}
If the immersion $\imm:\HH\to\Mink$ is of class $W^{2,p}_\loc$ and preserves the operators  $\naH,\KH, \LH, \MH$, then the following generalized Gauss and Codazzi equations are satisfied in the distributional sense: 
\begin{equation}
\label{gc}
\aligned
R_{\ hij}^k+K_{ih}L_j^k-K_{jh}L_i^k &=0 && \text{(Gauss)}, 
\\
\naH_i K_{jh}-\naH_j K_{ih}-K_{jh}M_i+K_{ih}M_j &=0 && \text{(Codazzi-1),}
\\
\naH_i L_j^k-\naH_j L_i^k -L_i^k M_j + L_j^k M_i &=0 && \text{(Codazzi-2),}
\\
\naH_i M_j-\naH_j M_i -K_{jh}L_i^h+K_{ih}L_j^h &=0 && \text{(Codazzi-3)}.
\endaligned
\end{equation}
\end{lemma}

\begin{proof} 
Let $\Ga_{ij}^k$ denote the Christoffel symbols associated with the connection $\naH$, so that
$$
\naH_{\d_i}{\d_j}=\Ga_{ij}^k \d_k. 
$$
Then, the definition of the operators $\naH,\KH, \LH, \MH$ shows that 
$$
\aligned 
\nabla_{\d_i}{\d_h} & = \Ga_{ih}^k \d_k - K_{ih}\rig,
\\
\nabla_{\d_i} \rig    & = L_i^k \d_k - M_i \rig,
\endaligned
$$
and the assumption that the immersion $\imm$ and rigging $\rigMink$ preserve the operators  $\naH,\KH, \LH, \MH$ shows that the function $f:=\imm\circ\ph^{-1}:\Om\to\r^{n+1}$ satisfies
$$
\aligned 
\frac{\d }{\d x^i}\Big( \frac{\d f}{\d x^h} \Big)& = \Ga_{ih}^k \frac{\d f}{\d x^k}  - K_{ih}\rigMink, \\
\frac{\d \rigMink }{\d x^i}  & = L_i^k \frac{\d f}{\d x^k}  - M_i \rigMink .
\endaligned
$$

Since the second derivatives of $\frac{\d f}{\d x^h}$ and $\rigMink$ commute, the above relations imply that 
$$
\aligned 
\frac{\d }{\d x^j}\Big( \Ga_{ih}^k \frac{\d f}{\d x^k}  - K_{ih}\rigMink  \Big)& = \frac{\d }{\d x^i}\Big( \Ga_{jh}^k \frac{\d f}{\d x^k}  - K_{jh}\rigMink  \Big),\\
\frac{\d }{\d x^j} \Big( L_i^k \frac{\d f}{\d x^k}  - M_i \rigMink \Big)   & =  \frac{\d }{\d x^i} \Big( L_j^k \frac{\d f}{\d x^k}  - M_j \rigMink \Big). 
\endaligned
$$
Hence, we find (the relations below should be understood in the distributional sense, against test functions in $\D(\Om)$, as in the proof of Lemma \ref{direct})
\begin{multline*}
\frac{\d \Ga_{ih}^k}{\d x^j}  \frac{\d f}{\d x^k}  - \frac{\d K_{ih}}{\d x^j}\rigMink  
+ \Ga_{ih}^k\Big( \Ga_{jh}^k \frac{\d f}{\d x^k}  - K_{jh}\rigMink  \Big) -K_{ih} \Big( L_j^k \frac{\d f}{\d x^k}  - M_j \rigMink \Big)\\
=\frac{\d \Ga_{jh}^k}{\d x^i}  \frac{\d f}{\d x^k}  - \frac{\d K_{jh}}{\d x^i}\rigMink  
+ \Ga_{jh}^k\Big( \Ga_{ih}^k \frac{\d f}{\d x^k}  - K_{ih}\rigMink  \Big) -K_{jh} \Big( L_i^k \frac{\d f}{\d x^k}  - M_i \rigMink \Big)
\end{multline*}
and 
\begin{multline*}
\frac{\d L_i^k}{\d x^j} \frac{\d f}{\d x^k}  - \frac{\d M_i}{\d x^j}  \rigMink 
+L_i^k\Big( \Ga_{jh}^k \frac{\d f}{\d x^k}  - K_{jh}\rigMink  \Big) - M_i \Big( L_j^k \frac{\d f}{\d x^k}  - M_j \rigMink \Big)\\
=
\frac{\d L_j^k}{\d x^i} \frac{\d f}{\d x^k}  - \frac{\d M_j}{\d x^i}  \rigMink 
+L_j^k\Big( \Ga_{ih}^k \frac{\d f}{\d x^k}  - K_{ih}\rigMink  \Big) - M_j \Big( L_i^k \frac{\d f}{\d x^k}  - M_i \rigMink \Big).
\end{multline*}

Using the fact that $\left\{\rigMink, \frac{\d f}{\d x^1},...,\frac{\d f}{\d x^n}\right\}$ 
is a basis in the tangent space $T\r^{n+1}$, it is easily seen that the last two equations are equivalent 
to the generalized Gauss and Codazzi equations of the lemma.
\end{proof}

We showed in the previous lemma that the generalized Gauss and Codazzi equations are necessary for the existence of an immersion $\imm$ and rigging $\rigMink$ preserving the operators $\naH,\KH, \LH, \MH$. We now show that these equations are also 
sufficient, at least as far as the local existence of $\imm$ and $\rigMink$ is concerned.

\begin{lemma}
\label{converse.local+} 
Suppose that $\naH,\KH, \LH, \MH$ are of class $L^p_\loc$, $p>n$, and satisfy the generalized Gauss and Codazzi equations \eqref{gc}. Consider any connected, simply 
connected, open subset $U\subset \HH$ that can de described by a single local chart. 
Then, there exists an immersion $\imm:U\to\Mink$ and a rigging $\rigMink:U\to T\Mink$, respectively of class $W^{2,p}_\loc(U)$ and $W^{1,p}_\loc(\imm(U))$, that preserve the operators $\naH,\KH, \LH, \MH$. 
Moreover, if $\si$ is an affine bijection of the Minkowski spacetime $\Mink$, then $\si\circ\imm$ and $\si_*\rigMink$ also satisfy this property.
\end{lemma}

\begin{proof}
Let $\ph:U\subset \HH \to \Om\subset \r^n$ be any local chart so that $\Om:=\ph(U)$ is connected and simply 
connected. Let $\Ga_{ij}^k$, $K_{ij}$, $L_j^k$, $M_i$ denote the components of the operators $\naH$, $\KH$, $\LH$, $\MH$ with respect to this chart and note that all these components belong to the space $L^p_\loc(\Om)$. 

Finding an immersion $\imm:U\to\Mink$ and a rigging $\rigMink: U\to T\Mink$ that preserve the operators  $\naH,\KH, \LH, \MH$ reduces to finding an immersion $f:\Om\to \r^{n+1}$ and a rigging $\rigMink:\Om\to T\r^{n+1}$ to the hypersurface $f(\Om)\subset\r^{n+1}$ such that 
\begin{equation}
\label{eq2+}
\aligned 
\frac{\d }{\d x^i}\Big( \frac{\d f}{\d x^h} \Big)& = \Ga_{ih}^k \frac{\d f}{\d x^k}  - K_{ih}\rigMink \\
\frac{\d \rigMink}{\d x^i} & = L_i^k \frac{\d f}{\d x^k}  - M_i \rigMink .
\endaligned
\end{equation}
This will be done in two stages. First, we solve the Pfaff system 
\begin{equation}\label{1+}
\frac{\d F}{\d x^i} = F \,A_i \quad \text{ a.e. in } \Om, \qquad 
A_i:=\mat{\Ga_{ih}^k}{L_i^k}{-K_{ih}}{-M_i},
\end{equation}
where in the definition of the matrices $A_i$ the row index is $k$ and the column index is $h$. 
In view of Theorem~\ref{pfaff}, 
this system has a solution $F\in W^{1,p}_\loc(\Om;\r^{(n+1)\times (n+1)})$, since the compatibility conditions 
$$
\frac{\d A_j}{\d x^i}-\frac{\d A_i}{\d x^j}=A_j A_i - A_i A_j  \qquad \text{ in } \D'(\Om,\r^{(n+1)\times (n+1)})
$$
are precisely 
equivalent to the Gauss and Codazzi equations \eqref{gc}. Moreover, the solution to this system is unique provided we impose an initial condition, say $F(x_\star)=F_\star\in \r^{(n+1)\times (n+1)}$ for some $x_\star\in\Om$ ($F_\star$ will be chosen later). 

Then, we solve the Poincar\'e system 
\begin{equation}
\label{1bis+}
\frac{\d f}{\d x^i} = F_i  \qquad \text{ in } \Om, 
\end{equation}
where $F_i$ is the $i$-th column vector of the matrix field $F$ that satisfies the system \eqref{1+}. This Poincar\'e system has a solution $f\in W^{2,p}_\loc(\Om,\r^{n+1})$ since  the compatibility conditions
$$
\frac{\d F_j}{\d x^i}= \frac{\d F_i}{\d x^j} 
$$
are satisfied. Indeed, they are equivalent to the equations $\Ga_{ij}^k=\Ga_{ji}^k$ and $K_{ij}=K_{ji}$ in view of 
equation \eqref{1+}; or the Christoffel symbols and the covariant components of the second fundamental form clearly satisfy these symmetry properties. 

To conclude the proof, we remark that $f$ is an immersion and the vector field $\rigMink:=F_{n+1}$ (that is, $\rigMink$ is the $(n+1)$-th column vector of the matrix field $F$ that satisfies the system \eqref{1+}) is transversal to $f(\Om)$ provided we choose the matrix $F_\star$ to be invertible. This is a consequence of the fact that the solution $F$  of the Pfaff system \eqref{1+} in invertible at every $x\in \Om$ if and only if $F$ is invertible at one single point; see S. Mardare \cite[Lemma 6.1]{sor-lpN}. The desired immersion is then defined by $\imm:=f\circ\ph:U\to\Mink$. 

Finally, if $\si$ is any affine bijection of the Minkowski spacetime $(\r^{n+1},\gMink)$, it is easily checked (in view of Equations \eqref{1+}) that $\si\circ\imm$ and $\si_*\rigMink$ also satisfy the conclusions of the lemma. 
\end{proof}

Lemma \ref{converse.local+} is a local existence result, in the sense that the immersion $\imm$ and rigging $\rigMink$ are  defined only on a subset $U\subset \HH$. But we have a good control of the size  of this neighborhood, since $U$ only needs to be connected, simply 
connected, and defined by a single local chart. As already remarked in the previous section, this will allow us to establish a global existence result when $\HH$ is simply 
connected. (See Lemma~\ref{converse.global} below.) 

We next prove that the immersion $\imm:\HH\to\Mink$ and rigging $\rigMink:\HH\to T\Mink$ defined by Lemma \ref{converse.local+} depend continuously (up to an affine bijection of the Minkowski space) on the operators $\naH,\KH, \LH,\MH$ . Note that if $\sigma:\Mink\to\Mink$ is an affine bijection of the Minkowski space, then $(\imm,\rigMink)$ and $(\sigma\circ\imm,\sigma_*\rigMink)$ share the same operators $\naH,\KH, \LH,\MH$. 

\begin{lemma}
\label{stability+}
Let $\naH,\KH, \LH,\MH$ and $\naHbis,\KHbis, \LHbis, \MHbis$ denote the operators induced on the hypersurface $\HH\subset \MM$ by the immersions $\imm,\immbis:\HH\to \Mink$ and riggings $\rigMink,\rigMinkbis:\HH\to T\Mink$, respectively. For any connected smooth open set $\Acal\Subset \MM$ and any $\ep>0$ such that 
$$
\max\big(\|(\naH,\KH,\LH,\MH)\|_{L^p(\Acal)}  , 
\|(\naHbis,\KHbis,\LHbis, \MHbis)\|_{L^p(\Acal)} \big)
\leq \frac1\ep, 
$$
there exists a constant $C=C(\ep,\Acal)$ such that, for some affine bijection $\sigma:\Mink\to\Mink$ of the Minkowski space, 
$$
\|\immbis-\sigma\circ\imm\|_{W^{2,p}(\Acal)}+\|\rigMinkbis -\sigma_*\rigMink\|_{W^{1,p}(\Acal)} 
\leq C \, \Big( \|(\naHbis,\KHbis,\LHbis, \MHbis) - (\naH,\KH,\LH,\MH)\|_{L^p(\Acal)}  \Big).
$$  
\end{lemma}

\begin{proof}
Let $U\Subset \MM$ be a connected smooth open set for which there exists a local chart $\ph:U\subset \HH\to \Om\subset\r^{n}$. Let $f:=\imm\circ\ph^{-1}$ and $\wt f:=\immbis\circ\ph^{-1}$. Define the matrix fields $F,\wt F:\Om\to \r^{(n+1)\times(n+1)}$ whose columns are respectively the vector fields $\frac{\d f}{\d x^1},...,\frac{\d f}{\d x^n},\rigMink$ and $\frac{\d \wt f}{\d x^1},...,\frac{\d \wt f}{\d x^n},\rigMinkbis$. As in the proof of Lemma~\ref{converse.local+}, we then have  
$$
\frac{\d F}{\d x^i} = F \,A_i \quad \text{ a.e. in } \Om, \qquad 
A_i:=\mat{\Ga_{ih}^k}{L_i^k}{-K_{ih}}{-M_i} \, ,
$$
and 
$$
\frac{\d \wt F}{\d x^i} = \wt F \,\wt A_i \quad 
\text{ a.e. in } \Om, \qquad 
\wt A_i:=\mat{\wt\Ga_{ih}^k}{\LHbis_i^k}{-\KHbis_{ih}}{-\MHbis_i} \, .
$$

Let $x_\star\in \Om$. Since the matrices $F(x_*)$ and  $\wt F(x_*)$ are invertible, the matrix $Q:=\wt F(x_*) (F(x_*))^{-1}\in \r^{(n+1)\times (n+1)}$ is well defined and is also invertible. Then the matrix fields $(QF)$ and $\wt F$ satisfy $(QF)(x^*)=\wt F(x_*)$ and 
$$
\frac{\d (QF)}{\d x^i} = (QF) \,A_i  \ \text{ and } \ \frac{\d \wt F}{\d x^i} = \wt F \,\wt A_i \quad 
\text{ a.e. in } \Om. 
$$
Then, the stability property for Pfaff systems stated in Section~\ref{prel} shows that there exists a constant $C=C(\ep,\Om)$ such that 
$$
\|\wt F -  QF\|_{W^{1,p}(\Om)} \leq C \, 
\sum_i \| \wt A_i- A_i\|_{L^p(\Om)}.
$$
This inequality in turn implies (thanks to Poincar\'e-Wirtinger's inequality) that there exists a vector $v\in\r^{n+1}$ such that 
$$
\| \wt f- (v+Qf)\|_{W^{2,p}(\Om)} \leq C \,  \|d \wt f- Q (d f)\|_{W^{1,p}(\Om)} .
$$
Noting that $\wt F= [(d \wt f)\  \rigMinkbis]$ and $QF=[Q(df) \ Q\rigMink]$ (the notation $[...]$ designates the matrix obtained by adjoining the columns of the matrices listed inside the brackets), we deduce from the last two inequalities that
$$
\aligned 
\| \wt f- (v+ Q f) \|_{W^{2,p}(\Om)} + 
\|\rigMinkbis - Q \rigMink\|_{W^{1,p}(\Om)}  
\leq C \,  
\sum_i \| \wt A_i- A_i\|_{L^p(\Om)} 
& \\ \leq C \,   
\sum_{ij}\|(\wt\Ga_{ij}^k, \KHbis_{ij}, \LHbis_i^k, \MHbis_i)-(\Ga_{ij}^k,K_{ij}, L_i^k, M_i)\|_{L^p(\Om)}. &
\endaligned
$$ 

Using the change of variables $x=\ph(p)$ in the last inequality above yields the inequality of the lemma over the set $\Acal=U$, with a mapping $\sigma:\Mink\to\Mink$ defined by $\sigma(y)=v+Qy$ for all $y\in \r^{n+1}$. Finally, to derive the desired inequality over the possibly larger set $\Acal$ it suffices to use the connectedness of $\Acal$ and the compactness of $\ov{\Acal}$.
\end{proof}

An immediate consequence of the previous lemma is the following uniqueness result. 

\begin{lemma}
\label{uniqueness+}
If the immersions $\imm,\immbis:\HH\to\Mink$ and riggings $\rigMink,\rigMinkbis$ are respectively of class $W^{2,p}_{loc}$ and $W^{1,p}_{loc}$, with $p>n$, and preserve the operators $\naH,\KH,\LH,\MH$, then for every connected component of $\HH$ there exists an affine bijection $\sigma:\Mink\to\Mink$ of the Minkowski spacetime such that $\immbis=\sigma\circ\imm$ and $\rigMinkbis=\sigma_*\rigMink$.
\end{lemma}

We are now in a position to establish a \emph{global} version of Lemma \ref{converse.local+}. 
The proof is similar to that of Lemma~\ref{converse.global} and is omitted.  
This also concludes the proof of Theorem~\ref{main-two}.  

\begin{lemma}
\label{converse.global+}
Suppose that $\HH$ is simply connected and that the operators $\naH, \KH, \LH, \MH$ are of class $L^p_\loc(\HH)$, $p>n$, and satisfy the Gauss and Codazzi equations \eqref{gc}. Then, there exists an immersion $\imm:\HH\to\Mink$ and a rigging $\rigMink:\HH\to T\Mink$, respectively of class $W^{2,p}_\loc$ and $W^{1,p}_\loc$, that preserve the operators $\naH, \KH, \LH, \MH$. 
\end{lemma}


\section{Immersion of a spacelike or timelike hypersurface} 
\label{hypersurface2}

In this section, we specialize the results of the previous section to hypersurfaces that are nowhere null.  
If $\HH\subset \MM$ is such a hypersurface, the normal vector field $\nH^\sharp$ is transversal to $\HH$ and therefore the metric $\gH$, induced on $\HH$ by the metric $g$ of the surrounding manifold $\MM$, is non-degenerate (see, e.g., LeFloch \& Mardare \cite[Theorem 6.1]{lfma}). For this reason, we need not prescribe a rigging along $\HH$, 
the projection on $\HH$ being made along $\nH^\sharp$. With the notation of the previous section, 
this is equivalent to choosing $\rig=\nH^\sharp$. Therefore, the results of the previous section apply to a spacelike or timelike hypersurface $\HH$, 
but can be simplified since the operators $\naH,\KH,\LH, \MH$ are now uniquely determined by the fundamental forms of the hypersurface $\HH$. How this simplification can be achieved is the subject of this section. 

Let $(\MM,g)$ be a Lorentzian manifold of dimension $n+1$.
Denote by $[X,Y]$ the Lie bracket of two vector fields, 
and by $\nabla$ the Levi-Civita connection induced by $g$.
Define the operator $\mbox{ }^\sharp:\theta\in T^*\MM\to \theta^\sharp\in T\MM$ by 
$$
\la \theta, Y\ra=g(\theta^\sharp, Y), \quad Y\in T\MM.
$$
Consider an oriented non-null hypersurface in $\HH\subset \MM$ and fix a unit normal form $\nH$ on $\HH$. Since $\HH$ is non-null, the metric field $\gH:T\HH\times T\HH \to\r$, also known as the first fundamental form of $\HH$, has either index zero (Riemannian metric) or index one (Lorentzian metric). The second fundamental form on $\HH$ is defined as in the previous section by 
$$
\KH:T\HH\times T\HH\to \r, \quad \KH(X,Y)=\la \nabla_X \nH, Y\ra, \qquad  X,Y\in T\HH.
$$
Then, the operators $\naH,\KH,\LH,\MH$ associated with rigging 
$\rig=\nH^\sharp$ are defined in terms of $\gH,\KH$ as stated in the following lemma. The proof is omitted.

\begin{lemma}
\label{identities}
(i) \, $\naH$ is the Levi-Civita connection induced by $\gH$, i.e., $\naH:T\HH\times T\HH\to T\HH$ is the unique operator defined by Koszul formula 
$$
\aligned
2 \, \gH(\nabla_XY,Z)=& X(\gH(Y,Z))+Y(\gH(X,Z))-Z(\gH(X,Y))\\
& -\gH(X,[Y,Z])-\gH(Y,[X,Z])+\gH(Z,[X,Y]), \quad X,Y\in T\HH.
\endaligned
$$

(ii) \, $\LH(X)=\KH(X,\cdot)^\sharp$ \quad for all $X\in T\HH$.

(iii) \, $\MH=0$. 
\end{lemma}

An immediate consequence of this lemma is that the generalized Gauss and Codazzi equations (see \eqref{gc}) 
reduce to the classical equations, as follows: the generalized Gauss equations coincide with the Gauss equations, 
the Codazzi-1 equations coincide with Codazzi equations, Codazzi-2 equations are equivalent to the Codazzi equations, and Codazzi-3 equations are equivalent to the equations $g^{hk}K_{ih}K_{jk}=g^{hk}K_{jh}K_{ik}$ expressing the symmetry of the ``third'' fundamental form of $\HH$. 

If $\HH$ is a hypersurface in $\MM$ and $\imm:(\HH,\gH)\to (\r^{n+1},\gMink)$ is an immersion into the Minkowski spacetime,  then the image $\HH'=\imm(\HH)$ is locally a hypersurface in $\r^{n+1}$. Therefore, there exists a smooth unit normal form $\nH':\HH\to T\r^{n+1}$ to the hypersurface $\HH'$ (uniquely defined up to its sign); this means that $\langle \nH'_p,\imm_* X_p\rangle=0$ for all $p\in\HH$ and $X\in T\HH$. For definiteness, we choose the sign of $\nH'$ to be that for which the immersion $\imm$ preserve the orientation. The second fundamental form of $\HH'$ is then defined as the pull-back on $\HH'$ of the two-covariant tensor field $\naMink \nH'$. 

Our objective in this section is to study whether there exists an immersion $\imm:\HH\to\Mink$ that preserves the fundamental forms $\gH,\KH$ of the hypersurface, that is, an immersion that satisfies 
$$
\imm^* \gMink=\gH  \text{ \ and \ } \imm^*K'=K.
$$

Let $\imm:(\HH,\gH)\to (\r^{n+1},\gMink)$ be an isometric immersion, that is, an immersion that satisfies $\imm^*\gMink=\gH$. If $\HH$ is nowhere null, then the hypersurface $\HH'=\imm(\HH)$ is also nowhere null (since the metric induced by $\gMink$ on $\HH'$ is non-degenerate). Therefore, the unit normal vector field ${\nH'}^\sharp:\HH\to T\r^{n+1}$ is transversal to $\HH'$, hence $\ell':={\nH'}^\sharp$ is a rigging in the sense stated in Section \ref{hypersurface1}. Then the operators $\naHMink,K',L',M'$ associated with the immersion $\imm$ and rigging ${\nH'}^\sharp$ are well defined (see Section \ref{hypersurface1}) and they satisfy the conclusions of Lemma \ref{identities}. As a consequence, \emph{an immersion $\imm:\HH\to\Mink$ preserves the fundamental forms of $\HH$ if and only if $(\imm,{\nH'}^\sharp)$ preserves the operators $\naHMink,K',L',M'$.}

In what follows, we may use local coordinates on the hypersurface: if $\ph:U\subset \HH\to \Om\subset \r^{n}$ denotes a local chart at $p\in\HH$, then $x^i$ denotes a set of Cartesian coordinates in $\Om$ and $\d_i$ denotes the vector field in $T\HH$ tangent to the coordinate line $x^i$. Note that the vector fields $\{\d_1,...,\d_n\}$ form a basis of the tangent space
 $T\HH$, while $\{\nH^\sharp,\d_1,...,\d_n\}$ form a basis of the  space $T\MM$. We denote $\gH_{ij}$, $K_{ij}$, $\Ga_{ij}^k$, $R_{\ hij}^k$ respectively the components in the local coordinates $x_i$ of $\gH$, $\KH$, $\naH$, $\Rm_{\gH}$, where $\Rm_{\gH}$ is the Riemann curvature tensor field associated with the metric $\gH$. Finally, let $(\gH^{hk}):=(\gH_{ij})^{-1}$ and $K^h_i:=\gH^{hk}K_{ki}$.

From Lemmas \ref{identities} and \ref{direct+}, we immediately deduce the following necessary conditions for the existence of an immersion preserving the fundamental forms. 

\begin{lemma}
\label{direct++}
If the immersion $\imm:\HH\to\Mink$ of class $W^{2,p}_\loc$, $p>n$, preserves the fundamental forms of $\HH$, then 
\begin{equation}
\label{gc+}
\aligned
R_{\ hij}^k+ K_{ih}K_j^k-K_{jh}K_i^k&=0 && \text{(Gauss)},\\
\naH_i K_{jh}-\naH_j K_{ih} &=0 && \text{(Codazzi)}.
\endaligned
\end{equation}
\end{lemma}

We next show that these equations are sufficient for the existence of a local immersion $\imm$. 

\begin{lemma}
\label{converse.local++}
Suppose that $\gH\in W^{1,p}_\loc(\HH)$ and $\KH\in L^p_\loc(\HH)$, $p>n$, satisfy the Gauss and Codazzi equations \eqref{gc+}. 
Consider any connected and simply connected open subset $U\subset \HH$ that can de described by a single local chart.
Then, there exists an immersion $\imm:U\to\Mink$ of class $W^{2,p}_\loc(U)$ that preserves the fundamental forms of the hypersurface. 
Moreover, if $\pi$ is a proper isometry of the Minkowski spacetime $\Mink$, then $\pi\circ\imm:\HH\to\Mink$ also preserves the fundamental forms of the hypersurface. 
\end{lemma}

\begin{proof}
The metric $\gH$ may be either Riemannian or Lorentzian. The proof is the same in both cases, 
except for the value of the parameter $\lambda=g(\nH,\nH)$ appearing below, which is equal to $-1$ if $\gH$ is Riemannian and to $1$ if $\gH$ is Lorentzian. 

Let $\ph:U\subset \HH \to \Om\subset \r^n$ be any local chart, so that $\Om:=\ph(U)$ is connected and simply 
connected. Fix a point $x_\star\in \Om$ and an invertible matrix $F_\star\in \r^{(n+1)\times (n+1)}$ that satisfies the relation (this choice will be explained later)
$$
F_\star^T \idMink F_\star =\mat{(\gH_{ij}(x_\star))}{0}{\quad 0}{\lambda},
$$ 
where $\idMink$ was defined in \eqref{defA}. 
The existence of such a matrix $F_\star$ is proved by Lemma \ref{matrixdecomp}. Since the matrix $(-F_\star)$ also satisfies the equation above, we may assume that $\det F_\star>0$. 

As in the proof of Lemma \ref{converse.local+}, there exists a unique matrix field $F$ belonging to
$W^{1,p}_\loc(\Om;\r^{(n+1)\times (n+1)})$ that satisfies the Pfaff system 
\begin{equation}\label{1++}
\aligned
& \frac{\d F}{\d x^i} = F \,C_i \quad \text{ a.e. in } \Om, 
\qquad C_i:=\mat{\Ga_{ih}^k}{K_i^k}{-K_{ih}}{0},\\
& F(x_\star) =F_\star. 
\endaligned 
\end{equation}
Let $F_i$ denote the $i$-th column vector field of the matrix field $F$. Again as in the proof of Lemma \ref{converse.local+}, 
there exists a vector field $f\in W^{2,p}_\loc(\Om,\r^{n+1})$, unique up to the addition of a constant vector field, that satisfies the Poincar\'e system 
\begin{equation}
\label{1bis++}
\frac{\d f}{\d x^i} = F_i  \qquad \text{ in } \Om. 
\end{equation}
Then one can see that $f$ is an immersion and satisfies (see \eqref{eq2+}):
\begin{equation}
\label{eq2++}
\aligned 
\frac{\d }{\d x^i}\Big( \frac{\d f}{\d x^h} \Big)& = \Ga_{ih}^k \frac{\d f}{\d x^k}  - K_{ih}\rigMink , \\
\frac{\d \rigMink}{\d x^i} & = K_i^k \frac{\d f}{\d x^k} ,
\endaligned
\end{equation}
where $\rigMink$ denotes the $(n+1)$-column vector field of the matrix field $F$. 

We now prove that $F$ satisfies 
$$
F^T(x) \idMink F(x) =\mat{(\gH_{ij}(x))}{0}{\quad 0}{\lambda}, \quad x\in \Om.
$$ 
By construction, this relation is satisfied at $x_\star$. Furthermore, on one hand, equation \eqref{1++} 
implies that 
$$
\frac{\d}{\d x^i}\left[F^T \idMink F\right] = C_i^T \left[F^T \idMink F\right] + \left[F^T \idMink F\right] C_i, 
$$
and, on the other hand, the definition of the Christoffel symbols $\Ga_{ij}^k$ shows that
$$
\frac{\d}{\d x^i}\mat{(\gH_{ij})}{0}{\quad 0}{\lambda} = C_i^T \mat{(\gH_{ij})}{0}{\quad 0}{\lambda} + \mat{(\gH_{ij})}{0}{\quad 0}{\lambda} C_i. 
$$
Therefore the uniqueness part of Theorem~\ref{pfaff} shows that 
$$
F^T(x) \idMink F(x) =\mat{(\gH_{ij}(x))}{0}{\quad 0}{\lambda}, \quad x\in \Om,
$$ 
since both satisfies the same Cauchy problem. 

Let us now prove that $\imm:=f\circ\ph:U\to\Mink$ satisfies the conclusions of the lemma. First, the above equation shows that $f$ is an isometric immersion; in other words, $\imm$ preserves the first fundamental form of the hypersurface $\HH$. Second, it shows that the $(n+1)$-column vector of $F$, denoted 
$\rigMink$, is orthogonal to the tangent space of the hypersurface $f(\Om)$ in the Minkowski spacetime and that $\gMink(\rigMink,\rigMink)=\lambda$. This implies that either $\rigMink={\nH'}^\sharp$, or $\rigMink=-{\nH'}^\sharp$. In fact, since  $\det F(x_\star)>0$ and $F$ is continuous (thanks to the Sobolev embedding $W^{1,p}_\loc(\Om)\subset \C^0(\Om)$ for $p>n$) on the connected set $\Om$, we have $\det F>0$ at every point of $\Om$; therefore $\rigMink={\nH'}^\sharp$. Combined with the equations \eqref{eq2++}, which are nothing but the classical Gauss and Weingarten equations on the hypersurface $f(\Om)$, this implies that $K_{ij}$ are the covariant components of the second fundamental form of $f(\Om)$. In terms of the immersion $\imm$, this means that $\imm$ preserves the second fundamental form of the hypersurface $\HH$. 
\end{proof}

Before extending the local immersion of Lemma \ref{converse.local++} to a 
global one, we need prove the uniqueness of such an immersion. In fact, we will establish a stronger result, namely that the immersion $\imm:\HH\to\Mink$ depends continuously on its fundamental forms.

\begin{lemma}
\label{stability++}
Let $(\gH,\KH)$ and $(\gHbis,\KHbis)$ denote the fundamental forms induced on the hypersurface $\HH\subset \MM$ by two immersions $\imm,\immbis:\HH\to \Mink$, respectively. For any connected smooth open set $\Acal\Subset \MM$ and any $\ep>0$, there exists a constant $C=C(\ep,\Acal)$ with the following property: if the fundamental forms $(\gH,\KH)$ and $(\gHbis,\KHbis)$ satisfy 
$$
\aligned 
& \min\big( |\det \gH|, |\det \gHbis|\big) 
\geq \ep, 
\\
& \max\big(\|\gH\|_{W^{1,p}(\Acal)},\|\KH\|_{L^p(\Acal)},\|\gHbis\|_{W^{1,p}(\Acal)},\|\KHbis\|_{L^p(\Acal)}\big) \leq \frac1\ep, 
\endaligned 
$$
then there exists proper isometries $\pi$ and $\wt\pi$ of the Minkowski space such that
$$
\|\wt\pi\circ\immbis-\pi\circ \imm\|_{W^{2,p}(\Acal)} 
\leq C \, \Big( \|\gHbis-\gH\|_{W^{1,p}(\Acal)} + \|\KHbis-\KH\|_{L^p(\Acal)} \Big). 
$$
\end{lemma}

\begin{proof}
We follow the proof of Lemma \ref{stability+}, save for the choice of the proper isometries $\pi$ and $\wt\pi$. 
Let $U\Subset \MM$ be a connected smooth open set for which there exists a local chart $\ph:U\subset \HH\to \Om\subset\r^{n}$. Let $f:=\imm\circ\ph^{-1}$ and $\wt f:=\immbis\circ\ph^{-1}$. Define the matrix fields $F,\wt F:\Om\to \r^{(n+1)\times(n+1)}$ whose columns are respectively the vector fields $\frac{\d f}{\d x^1},...,\frac{\d f}{\d x^n},{\nH'}^\sharp$ and $\frac{\d \wt f}{\d x^1},...,\frac{\d \wt f}{\d x^n},{\wt\nH'}^\sharp$. As in the proof of Lemma~\ref{converse.local++}, we then have  
$$
\frac{\d F}{\d x^i} = F \,C_i \quad \text{ a.e. in } \Om, \qquad 
C_i:=\mat{\Ga_{ih}^k}{K_i^k}{-K_{ih}}{0} \, ,
$$
and 
$$
\frac{\d \wt F}{\d x^i} = \wt F \,\wt C_i \quad 
\text{ a.e. in } \Om, \qquad 
\wt C_i:=\mat{\wt\Ga_{ih}^k}{\KHbis_i^k}{-\KHbis_{ih}}{0} \, .
$$

Let $x_\star\in \Om$. Since $\imm$ and $\immbis$ are isometric immersions (i.e., they preserve the first fundamental form) and ${\nH'}^\sharp, {\wt\nH'}^\sharp$ are unit normal fields, we have on one hand 
$$
F(x_\star)^T \idMink F(x_\star) =\mat{(\gH_{ij}(x_\star))}{0}{\quad 0}{\lambda}, 
\qquad 
\wt F(x_\star)^T \idMink \wt F(x_\star) =\mat{\big(\gHbis_{ij}(x_\star)\big)}{0}{\quad 0}{\lambda}.
$$ 
On the other hand, Lemma \ref{matrixdecomp} shows that there exists matrices $E_\star,\wt E_\star$ such that 
$$
E_\star^T \idMink E_\star = \mat{(\gH_{ij}(x_\star))}{0}{\quad 0}{\lambda}, 
\qquad
\wt E_\star^T \idMink \wt E_\star=\mat{\big(\gHbis_{ij}(x_\star)\big)}{0}{\quad 0}{\lambda},
$$
and 
\begin{equation}
 \label{newest}
|\wt E_\star - E_\star|\leq C \big| (\gHbis_{\si\ta}(x_\star)) - (\gH_{\si\ta}(x_\star))\big| \leq \|(\gHbis_{\si\ta}) - (\gH_{\si\ta})\|_{W^{1,p}(\Om)}. 
\end{equation}
As explained in the proof of Lemma \ref{converse.local++}, we may assume that $\det E_\star>0$ and $\det \wt E_\star>0$. Let $Q:=E_\star F(x_\star)^{-1}$ and $\wt Q:= \wt E_\star (\wt F(x_\star))^{-1}$ and note that they are Minkowski-orthogonal matrices with positive determinant. 

The definition of the matrices $Q$ and $\wt Q$ implies that the matrix fields $(QF)$ and $(\wt Q\wt F)$ satisfy
$$
\aligned
&
\frac{\d (QF)}{\d x^i} = (QF) \,C_i \text{ \ a.e. in } \Om, \qquad (QF)(x_\star)=E_\star,\\
& 
\frac{\d (\wt Q\wt F)}{\d x^i} = (\wt Q\wt F) \,\wt C_i   \text{ \ a.e. in } \Om, \qquad (\wt Q \wt F)(x_\star)=\wt E_\star.
\endaligned
$$
Then, in  view of the stability property for Pfaff systems stated in Section~\ref{prel}, there exists a constant $C=C(\ep,\Om)$ such that 
$$
\|\wt Q\wt F -  QF\|_{W^{1,p}(\Om)} \leq C \, \Big( |\wt E_\star- E_\star | + \sum_i \| \wt C_i- C_i\|_{L^p(\Om)} \Big). 
$$
Using inequality \eqref{newest} and the definition of matrices $C_i, \wt C_i$, we next obtain 
$$
\|\wt Q\wt F -  QF\|_{W^{1,p}(\Om)}
 \leq C \, \Big( \|(\gHbis_{\si\ta}) - (\gH_{\si\ta})\|_{W^{1,p}(\Om)} +\| (\wt K_{ij})- (K_{ij})\|_{L^p(\Om)} \Big).
$$
Noting that $\wt Q\wt F= [\wt Q(d \wt f)\  \wt Q{\wt\nH'}^\sharp]$ and $QF=[Q(df) \ Q{\nH'}^\sharp]$ (the notation $[...]$ designates the matrix obtained by adjoining the columns of the matrices listed inside the brackets), we deduce from the above inequality that
$$
\|d( \wt Q \wt f -  Qf)\|_{W^{1,p}(\Om)}
 \leq C \, \Big( \|(\gHbis_{\si\ta}) - (\gH_{\si\ta})\|_{W^{1,p}(\Om)} +\| (\wt K_{ij})- (K_{ij})\|_{L^p(\Om)} \Big).
$$ 
This inequality in turn implies (thanks to Poincar\'e-Wirtinger's inequality) that 
$$
\| (\wt v+\wt Q\wt f)- (v+Qf)\|_{W^{2,p}(\Om)} \leq C \,  \Big( \|(\gHbis_{\si\ta}) - (\gH_{\si\ta})\|_{W^{1,p}(\Om)} +\| (\wt K_{ij})- (K_{ij})\|_{L^p(\Om)} \Big), 
$$
where $v=-Qf(x_\star)$ and $\wt v=- \wt Q \wt f(x_\star)$. Since the matrices $Q$ and $\wt Q$ are proper Minkowski-orthogonal, the mappings $\pi: y\in\r^{n+1}\mapsto v+ Qy\in\r^{n+1}$ and  $\wt \pi: y\in\r^{n+1}\mapsto \wt v+ \wt Qy\in\r^{n+1}$ are proper isometries of the Minkowski spacetime $\Mink$. Finally, letting $x=\ph(p)$, $p\in U$, in the above inequality shows that
$$
\|\wt\pi\circ \immbis-\pi\circ\psi\|_{W^{2,p}(U)} \leq C \, \big( \|\gHbis-\gH\|_{W^{1,p}(U)} + \|\KHbis-\gH\|_{L^p(U)}\big).
$$
This inequality still holds when $U$ is replaced with the possibly larger set $\Acal$ since $\Acal$ is connected and $\overline{\Acal}$ is compact.
\end{proof}

An immediate consequence of the previous lemma is the following uniqueness result. 

\begin{lemma}
\label{uniqueness++}
If the immersions $\imm,\immbis:\HH\to\Mink$ of class $W^{2,p}_{loc}$, $p>n$, have the same fundamental forms, then for every connected component of $\MM$ there exists a proper isometry $\ta$ of the Minkowski space such that $\immbis=\ta\circ\imm$. 
\end{lemma}

We are now in a position to establish a \emph{global} version of Lemma \ref{converse.local+}. 
The proof is similar to that of Lemma \ref{converse.global} and is omitted.  

\begin{lemma}
\label{converse.global++}
Suppose that $\HH$ is simply 
connected and that $(\gH, \KH)$ are of class $W^{1,p}_\loc(\HH)\times L^p_\loc(\HH)$, $p>n$, and satisfy the Gauss and Codazzi equations \eqref{gc+}. Then, there exists an immersion $\imm:\HH\to \Mink$ of class $W^{2,p}_\loc(\HH)$ that preserves the fundamental forms of $\HH$. 
\end{lemma}



\begin{thebibliography}{99}

\newcommand{\auth}{\textsc}
\newcommand{\jour}{} 
\newcommand{\book}{\emph}

\bibitem{AMR}
\auth{Abraham R., Marsden J.E., Ratiu T.,} 
\book{Manifolds, tensor analysis, and applications,}  
Springer Verlag, New York, 1988. 

\bibitem{ADAMS}
\auth{Adams R.A.,} 
\book{Sobolev spaces,} Academic Press, New York, 1975. 

\bibitem{Ciarlet} 
\auth{Ciarlet P.G.,}
\book{An introduction to differential geometry with applications to elasticity,}
Springer Verlag, Dordrecht (2005).

\bibitem{hawi}
\auth{Hartman P. and Wintner A.},  
On the fundamental equations of differential geometry, 
\jour{Amer. J. Math.} 72 (1950) 757--774.

\bibitem{lfma}
\auth{LeFloch P.G. and Mardare C.,} 
Definition and stability of Lorenzian manifolds with distributional curvature, 
\jour{Port. Math.} (2007), 535--573.  

\bibitem{sor-infty}
\auth{Mardare S.,} 
On isometric immersions of a Riemannian space with little regularity,  
\jour{Analysis Appl.} 2 (2004), 193--226. 

\bibitem{sor-lp} 
\auth{Mardare S.}, 
On Pfaff systems with $L^p$ coefficients and their applications in differential geometry, 
\jour{J. Math. Pures Appl.} 84 (2005) 1659--1692.

\bibitem{sor-lpN}
\auth{Mardare S.,} 
On systems of first-order linear partial differential equations with $L^p$ coefficients,
\jour{Adv. Diff. Equ.} 12 (2007), 301--360.


\end{thebibliography}
\end{document}